\documentclass[10pt, english]{amsart}

\usepackage{amsmath,amssymb,enumerate}

\usepackage[T1]{fontenc}
\usepackage[all,cmtip]{xy}

\usepackage{babel}
\usepackage{amstext}
\usepackage{amsmath}
\usepackage{amsfonts}
\usepackage{latexsym}
\usepackage{ifthen}

\usepackage{xypic}
\xyoption{all}
\newcommand{\PN}{{\mathbb P}}

\newcommand{\rk}{{\rm rk}}

\newcommand{\sJ}{{\mathcal J}}

\newtheorem{lemma1}{}[section]

\newenvironment{lemma}{\begin{lemma1}{\bf Lemma.}}{\end{lemma1}}
\newenvironment{example}{\begin{lemma1}{\bf Example.}\rm}{\end{lemma1}}

\newenvironment{theorem}{\begin{lemma1}{\bf Theorem.}}{\end{lemma1}}

\newenvironment{proposition}{\begin{lemma1}{\bf Proposition.}}{\end{lemma1}}
\newenvironment{corollary}{\begin{lemma1}{\bf Corollary.}}{\end{lemma1}}
\newenvironment{remark}{\begin{lemma1}{\bf Remark.}\rm}{\end{lemma1}}
\newenvironment{remarks}{\begin{lemma1}{\bf Remarks.}\rm}{\end{lemma1}}
\newenvironment{definition}{\begin{lemma1}{\bf Definition.}}{\end{lemma1}}
\newenvironment{notation}{\begin{lemma1}{\bf Notation.}}{\end{lemma1}}

\newenvironment{setup}{\begin{lemma1}{\bf Setup.}}{\end{lemma1}}

\newenvironment{conjecture}{\begin {lemma1}{\bf Conjecture.}}{\end{lemma1}}

\newenvironment{problem}{\begin{lemma1}{\bf Problem.}}{\end{lemma1}}

\newenvironment{the local obstruction - setup}{\begin{lemma1}{\bf The local obstruction - setup.}}{\end{lemma1}}

\newenvironment{remark*}{{\bf Remark.}}{}
\newenvironment{remarks*}{{\bf Remarks.}}{}
\newenvironment{example*}{{\bf Example.}}{}
\newenvironment{assumption*}{{\bf Assumption.}}{}

\newcommand{\R}{\ensuremath{\mathbb{R}}}
\newcommand{\Q}{\ensuremath{\mathbb{Q}}}
\newcommand{\Z}{\ensuremath{\mathbb{Z}}}
\newcommand{\C}{\ensuremath{\mathbb{C}}}
\newcommand{\N}{\ensuremath{\mathbb{N}}}
\newcommand{\PP}{\ensuremath{\mathbb{P}}}

\usepackage{eurosym}

\newcommand{\merom}[3]{\ensuremath{#1:#2 \dashrightarrow #3}}

\newcommand{\holom}[3]{\ensuremath{#1:#2  \rightarrow #3}}
\newcommand{\fibre}[2]{\ensuremath{#1^{-1} (#2)}}

\makeatletter
\ifnum\@ptsize=0  \fi \ifnum\@ptsize=2  \fi 

\newcommand\sE{{\mathcal E}}

\newcommand\sF{{\mathcal F}}
\newcommand\sG{{\mathcal G}}

\newcommand\sI{{\mathcal I}}
\newcommand\sT{{\mathcal T}}
\newcommand\sL{{\mathcal L}}

\newcommand\sO{{\mathcal O}}

\DeclareMathOperator*{\pic}{Pic}
\DeclareMathOperator*{\sing}{sing}

\DeclareMathOperator*{\nons}{nons}

\usepackage{xcolor}
\usepackage{hyperref}

\newcommand{\PNM}{\ensuremath{\mathbb P(T_M)}}
\newcommand{\NEM}{\overline{\mbox{NE}}(M)}

\author{Andreas H\"oring}
\author{Thomas Peternell}

\address{Andreas H\"oring, Universit\'e C\^ote d'Azur, CNRS, LJAD, France, Institut universitaire de France}
\email{Andreas.Hoering@univ-cotedazur.fr}

\address{Thomas Peternell, Mathematisches Institut, Universit\"at Bayreuth, 95440 Bayreuth, 
Germany}
\email{thomas.peternell@uni-bayreuth.de}


\title{Stein complements in compact K\"ahler manifolds} 
\date{November 5, 2021} 

\begin{document}

\begin{abstract} Given a projective or compact K\"ahler manifold $X$ and a (smooth) hypersurface $Y$,  we study conditions under which $X \setminus Y$ could be
Stein. We apply this in particular to the case when $X $ is the projectivization of the so-called canonical extension of the tangent bundle $T_M$ of a projective manifold $M$
with $Y$ being the projectivization of $T_M$ itself. 

\end{abstract} 

\maketitle

\section{Introduction} 

\subsection{Motivation}

Let $M$ be a compact K\"ahler manifold, and let $\alpha \in H^1(M, \Omega_M)$ be a K\"ahler class. This defines an extension of vector bundles, 
$$
0 \to \sO_M \to V^{\alpha} \to T_M \to 0 
$$
and therefore an embedding $\mathbb P(T_M) \subset \mathbb P(V^{\alpha})$. We then consider the complement
$$
Z_M := Z_M^\alpha :=   \mathbb P(V^{\alpha}) \setminus \mathbb P(T_M).
$$
Following  Greb and Wong \cite{GW20} we call $Z_M$ a {\it canonical extension of $M$}.
They showed that $Z_M$ does contain any compact subvarieties and it is Stein 
if $M$ is a torus or has a K\"ahler metric of non-negative holomorphic bisectional curvature. In view of their results one is tempted to make the following conjecture; already formulated as a question in \cite{GW20}, 

\begin{conjecture} \label{conjecture-GW} 
Let $M$ be a compact K\"ahler manifold, and let $Z_M$ be a canonical extension defined by some K\"ahler class on $M$. Then the following holds:
\begin{itemize}
\item $Z_M$ is Stein if and only if $T_M$ is nef.
\item $Z_M$ is affine\footnote{We say that a complex manifold $U$ is affine if there exists an affine variety $U_{\rm alg}$ such that $U$ is biholomorphic to the analytification of $U_{\rm alg}$.} if and only if $T_M$ is nef and big. 
\end{itemize}
\end{conjecture}

Let us first focus on the affine version of the conjecture. 
We will use the basepoint-free theorem to show one implication:

\begin{theorem} \label{prop:nefFano} Let $M$ be a compact K\"ahler manifold such that $T_M$ is nef and big. 
Then $Z_M$ is affine for any K\"ahler class  $M$.
\end{theorem}

For the other implication observe first that if $Y:= \PP(T_M) \subset \PP(V^\alpha)=:X$ is the embedding defined above,
the normal bundle $N_{Y/X}$ identifies to the tautological bundle $\sO_{\PP(T_M)}(1)$. Let us recall that Greb and Wong \cite[Prop.4.2]{GW20} used Goodman's theorem \cite{Go69} to show that if $X$ is projective and $X \setminus Y$ is affine, the normal bundle $N_{Y/X}$ is big. Thus the other implication would essentially reduce to a conjecture of Goodman:

\begin{conjecture} \cite[II.V.2]{Har70} \label{conjecture-goodman}
Let $X$ be a compact algebraic manifold, and let $Y \subset X$ be a prime divisor. If $X \setminus Y$ is affine,
then $N_{Y/X}$ is nef.
\end{conjecture}

The case where $X$ is a surface is of course well-known, see Proposition \ref{prop:one}: if $X \setminus Y$ is affine (or more generally Stein), then $N_{Y/X} $ is nef, i.e., $Y^2 \geq 0$. 
We construct however a threefold that is a counterexample to Conjecture \ref{conjecture-goodman} (but not to Conjecture \ref{conjecture-GW}):

\begin{proposition} \label{prop:goodman}
There exists a smooth projective threefold $X$ containing a smooth connected hypersurface $Y$
such that $X \setminus Y$ is affine and $N_{Y/X}$ is not nef.
\end{proposition}

In this paper we prove a weak version of Goodman's conjecture. We then show that in the geometric setting of Conjecture \ref{conjecture-GW} one can obtain much stronger results.

\subsection{General Stein complements}

Let $X$ be a compact complex K\"ahler manifold, and let $Y \subset X$ be an irreducible hypersurface. 
Generalising Conjecture \ref{conjecture-GW} to this setup we consider the following problem:

\begin{problem} \label{P1} Let $X$ be a compact K\"ahler manifold, and let $Y \subset X$ be an irreducible (smooth) hypersurface. Give necessary and sufficient criteria such $X \setminus Y$ is affine or Stein. 
\end{problem} 

Our first main result is a necessary condition for smooth hypersurfaces:

\begin{theorem} \label{thm:mainone} 
Let $X$ be a compact K\"ahler manifold of dimension $n$. Let $Y \subset X$ be a smooth hypersurface such that $X \setminus Y$ is Stein.
Then the normal bundle $N_{Y/X}$ is pseudoeffective.
\end{theorem} 

The restriction to smooth hypersurfaces is due to the fact that our main technical ingredient, a result of Kosarew and the second named author relating analytic and algebraic functions
on the complement, see Proposition \ref{prop:alg-an}, is only known in the smooth case.
We then combine this result with Yang's recent partial converse 
Andreotti-Grauert theorem \cite[Thm.1.5]{Yan19}.
Proposition \ref{prop:goodman} shows that Theorem \ref{thm:mainone} is essentially optimal: in our example the normal bundle $N_{Y/X}$ is pseudoeffective, but not nef in codimension one.
 In view of these results it is clear that the properties of complements are of a birational nature, therefore we study in Subsection \ref{subsectionmori} how Mori theory can be used to obtain a better understanding of the situation.
 
\subsection{Canonical extensions}

As a direct consequence of Theorem \ref{thm:mainone} one obtains:

\begin{corollary} \label{cor:extension} 
Let $M$ be a compact K\"ahler manifold.
Assume that  $Z_M$ is Stein for some K\"ahler class.
Then the tangent bundle $T_M$ is pseudoeffective.
\end{corollary}

While Proposition \ref{prop:goodman} shows that Theorem \ref{thm:mainone} can't be improved in general, it seems possible that Conjecture \ref{conjecture-GW} has a positive answer:
in fact, as already mentioned, Greb and Wong observed that $Z_M$ does not contain subvarieties of positive dimension \cite[Prop. 2.7]{GW20}. Thus, $Z_M$ is Stein if and only $Z_M$ is holomorphically convex. Using classical results about the holomorphic convexity of certain complements,
we show in Theorem \ref{theorem-contractions} that the Stein property leads to strong restrictions on divisorial Mori contractions. 
In low dimension this already gives a rather complete picture: 

\begin{corollary} \label{corollary-Mori-low-dimension} Let $M$ be a compact K\"ahler manifold of dimension at most three.
Assume that  $Z_M$ is Stein for some K\"ahler class. Then $M$ does not admit a birational Mori contraction. 
Thus, either $K_M$ is nef or $M$ is a Mori fibre space. 
\end{corollary}

The scope of our technique is not necessarily limited to divisorial Mori contractions:

\begin{proposition} \label{proposition-fourfold-flip} Let $M$ be a smooth compact K\"ahler fourfold.
Assume that  $Z_M$ is Stein for some K\"ahler class. Then $M$ does not admit
a small Mori contraction.
\end{proposition}

In a different direction we can use stability arguments to exploit the property that the tangent bundle is pseudoeffective:

\begin{proposition}  \label{proposition-gen-ample}
Let $M$ be a projective manifold of dimension $n$ such that $T_M$ is pseudoeffective. 
If $M$ is not uniruled, there exists exists a decomposition $$T_M \simeq \sF \oplus \sG,$$ where
$\sF \neq 0$ and $\sG$ are integrable subbundles such that $c_1(\sF)=0$.
\end{proposition}

In particular, by \cite[Cor.11]{Pe11} the manifold $M$ is not of general type and $\Omega_M$ is not generically ample.
In view of this result and \cite[Thm.1.6]{HP19} we expect the following to be true:
 
\begin{conjecture} \label{conjecture:torus} Let $M$ be a compact K\"ahler manifold that is not uniruled.
Assume that  $Z_M$ is Stein for some K\"ahler class.
Then $M$ is an \'etale quotient of a torus.
\end{conjecture} 

This statement would of course be a consequence of the first part of Conjecture \ref{conjecture-GW}.
Conjecture \ref{conjecture:torus} is obviously related to a conjecture of Pereira and Touzet on the algebraic integrability
of foliations with trivial first Chern class \cite[6.5]{PT13}. We will use some recent results from foliation theory to confirm it
for projective manifolds of low dimension:

\begin{theorem} \label{theorem-not-uniruled}
Let $M$ be a projective manifold of dimension at most three that is not uniruled. 
Assume that  $Z_M$ is Stein for some K\"ahler class.
Then $M$ is an \'etale quotient of an abelian variety.
\end{theorem}

For uniruled manifolds the situation is more complicated, see the discussion in Section \ref{sectionsurfaces}.
For surfaces we can summarise our results as follows:

\begin{theorem} 
\label{thm:maintwo} 
Let $M$ be a smooth projective surface.
Assume that  $Z_M$ is Stein for some K\"ahler class.
Then one of the following holds:
\begin{itemize}
\item $M$ is an \'etale quotient of a torus;
\item $M$ is rational homogeneous, i.e., $M = \mathbb P_2$ or $M = \mathbb P_1 \times \mathbb P_1$.
\item $M$ is a ruled surface over a curve of genus $B$ at least one. In case $g(B) \geq 2$, $M$ is given by a semi-stable vector bundle.
\end{itemize}
The canonical extension $Z_M$ is affine if and only if $M = \mathbb P_2$ or $M = \mathbb P_1 \times \mathbb P_1$.
\end{theorem} 

We expect actually more to be true. First, we should always have $g(B) \leq 1$. If $M$ is a ruled surface over an elliptic curve, given by a vector bundle $\sE$, then $\sE$ should be semi-stable; the converse does hold. Most of the arguments in the proof of  Theorem \ref{thm:maintwo} are valid for K\"ahler surfaces, but there are additional difficulties arising from non-projective elliptic bundles over curves of higher genus (cf. Remark \ref{remark-elliptic-bundles}).

\subsection{Future directions}

While we have seen above that the existence of a canonical extension $Z_M$
leads to numerous restriction on the geometry of $M$, it is quite difficult
to construct non-trivial examples
of projective manifolds such that $Z_M$ is Stein or even affine.
Theorem \ref{prop:nefFano} should have the following extension which is affirmed positively in dimension two by Theorem \ref{thm:maintwo}

\begin{conjecture} \label{conj:nefFano2} Let $M$ be a compact K\"ahler manifold with nef tangent bundle $T_M$. Then $Z_M$ is  affine if and only if $M$ is Fano. 
\end{conjecture}

Theorem \ref{thm:mainone} depends on Proposition \ref{prop:alg-an} whose proof is highly transcendental. In Section \ref{section-algebraic} we present a weaker version with a simple purely algebraic 
proof. This version is almost sufficient to prove Theorem \ref{thm:maintwo}, with some  K3-surfaces $M$ left over, namely when $\rho(M) \geq 2$ or when $M$ does not contain a nodal 
rational curve. It also leads to the following 

\begin{conjecture} \label{conjecture-gen-nef}
 Let $M$ be a simply connected projective manifold of dimension $n$ with $K_M \simeq \sO_M$. 
Then the tautological class $\zeta_M$ on $\PP(T_M)$ is not generically nef, i.e.,
$$\zeta_M \cdot H_1 \cdot \ldots \cdot H_{2n-2} < 0$$
for some ample divisors $H_j$ on $\PP(T_M)$. 
\end{conjecture}

It is known that $\zeta_M$ is not pseudoeffective, \cite[Thm.1.6]{HP19}, but the difference between pseudoeffectivity and generic nefness is quite significant.

{\bf Acknowledgements.} For very useful discussions we thank B. Claudon, F. Gounelas, D.Greb, A. Sarti and M.Zaidenberg.  The first-named author thanks the Institut Universitaire de France and the A.N.R. project Foliage (ANR-16-CE40-0008) for providing excellent working conditions.

\section{Preliminaries} 

We work over the complex numbers, for general definitions we refer to \cite{Har77, Kau83}. 
Complex spaces and varieties will always be supposed to be irreducible and reduced. 

We use the terminology of \cite{Deb01, KM98}  for birational geometry and notions from the minimal model program. We follow \cite{Laz04} for algebraic notions of positivity, in particular
the terminology of $\mathbb Q$-twisted bundles as explained in \cite[Ch.6.2]{Laz04}.
For positivity notions in the analytic setting, we refer to \cite{Dem12}.

\begin{notation}  Let $X$ be a normal projective variety,
and by $N_1(X)$ the $\R$-vector space of curves modulo numerical equivalence.

We denote by $\overline{NM}(X) \subset N_1(X)$ the closed cone spanned by (classes) of movable curves. By \cite{BDPP13}, this is the dual cone to the pseudoeffective cone.  

We denote by $\overline {\mathcal {CI}}(X) \subset N_1(X)$ the closed cone spanned by (classes of) general complete intersection curves $C = H_1 \cap \ldots \cap H_n$
of very ample divisors $H_j$. If the $H_j$ have sufficiently large degree and if $C$ is general, then $C$ is called a MR-curve (MR for Mehta-Ramanathan). \\
\end{notation}

\begin{notation}
Let $M$ be a complex space, and let $V \rightarrow M$ be a vector bundle over $M$.
We denote by $\holom{\pi}{\PP(V)}{M}$ the projectivisation of $V$ and by 
$\zeta_{\PP(V)}:=c_1(\sO_{\PP(V)}(1))$ the tautological class on $\PP(V)$.
\end{notation}

When the situation is clear, we will use the simplified notation $\zeta:= \zeta_{\PP(V)}$.
Since the terminology varies in the literature, let us recall:

\begin{definition}
    Let $M$ be a normal compact complex space, and let $V \rightarrow M$ be a vector bundle. We say that $V$ is pseudoeffective if the tautological
    class $\zeta_{\PP(V)}$ is pseudoeffective. The vector bundle is big if
    $\zeta_{\PP(V)}$ is big.
\end{definition}

\begin{definition} \label{definitiongennef} \begin{itemize}
\item Let $X$ be a normal projective variety.
An $\R$-divisor class on $X$ is generically nef, if it is non-negative on $\overline {\mathcal {CI}}(X) \subset N_1(X)$.
\item Let $M$ be a normal projective variety, and let $V \rightarrow M$ be a vector bundle.
The vector bundle $V$ is generically nef (resp. generically ample), if its restriction to all MR-curves is nef (resp. ample).
\end{itemize}
\end{definition} 

\begin{remark*}
If $X$ is a smooth projective surface, the cone $\overline {\mathcal {CI}}(X)$ coincides with the nef cone, so a divisor class on $X$ is generically nef if and only if it is pseudoeffective.

Let $V$ be a generically nef vector bundle on a higher-dimensional manifold. Then $\zeta_{\PN(V)}$ might not be generically nef, even if $V$ is semistable for any polarisation (cf. Proposition \ref{prop:K3} for an example).
\end{remark*}

\begin{proposition} \label{prop:curve} 
Let $C$ be a smooth compact curve, and let $V$ be a semistable vector bundle on $C$. Assume that $\zeta_{\PN(V)} $ is generically nef. Then $V$ is nef.
\end{proposition}

\begin{proof} Let $F$ be a fiber of $\PN (V) \to C$ and $\ell$ a line in $F$. Let $r$ be the rank of $V$ and set $\mu = \frac{c_1(V)}{r}$. 
By  \cite[1.2]{Fu11}, the cone $\overline{NE}(\PN (V)) $ is spanned by $(\zeta - \mu F)^{r-1}$ and $\ell$. Since $V$ is semistable, the class $\zeta - \mu F$ is nef, hence 
 $$ 
 \overline{NE}(\PN (V))  = \overline {\mathcal {CI}}(\PN (V)).
 $$
 Since $\zeta $ is generically nef, $\zeta$ is therefore nef, hence $V$ is nef. 
\end{proof} 

\begin{remark*} {\rm Proposition \ref{prop:curve} is wrong if the bundle is not stable; in fact $\zeta_{\PN(V)}$ might not be pseudoeffective.  For example, consider 
$$
V = \sO_{\PP^1}(-1) \oplus \sO_{\PP^1}(-2) \oplus \sO_{\PP^1}(-3)
$$ 
on $C = \PN_1$. 
Let $\zeta = \zeta_{\PN(V)}$ and $F$ a fiber of $\PN(V) \to C$. Then the nef cone of $\PN(V)$ is generated by $F$ and 
$\zeta + 3F$. Since 
$$ \zeta \cdot F^2 = 0 =  \zeta \cdot (\zeta + 3F)^2, \ \zeta \cdot F \cdot (\zeta + 3F) = 1,
$$
the tautological class $\zeta$ is generically nef. 
}
\end{remark*} 

In the example above $V^*$ is ample although $\zeta_{\PP(V)}$ is generically nef. However we have the following:

\begin{lemma} \label{lemmapseffample}
Let $V$ be a vector bundle over a projective variety $M$ such that $\zeta_{\PP(V)}$ is pseudoeffective.
Then $V^*$ is not ample.
\end{lemma}

\begin{proof}
If $V^*$ is ample there exists an ample $\Q$-divisor $A$ such that the twist $V^* \otimes A^*$ is still ample.
In particular, by \cite[Ex.6.1.4]{Laz04} one has $H^0(M, S^d (V \otimes A)) = 0$ for all sufficiently divisible $d \in \N$. Yet this implies that $\kappa(\PP(V), \zeta_{\PP(V)} + \pi^* A) = -\infty$, hence $\zeta_{\PP(V)}$ is not pseudoeffective.
\end{proof}

\begin{definition} Let $X$ be a projective manifold, and let $Y \subset X$ be an analytic set. Let $\sF$ be a coherent sheaf on $X$. 
Then we define local cohomology groups $H^q_{[Y]}(X,\sF) $ and $H^q_Y(X,\sF)$ to obtain long exact sequences 
$$  \ldots \to H^{q-1}(X \setminus Y, \sF)_{\rm alg}   \to  H^q_{[Y]}(X,\sF) \to H^q(X,\sF) \to H^{q}(X \setminus Y, \sF)_{\rm alg}  \to \ldots $$
for algebraic cohomology 
and 
$$  \ldots \to H^{q-1}(X \setminus Y, \sF)_{\rm an}   \to  H^q_{Y}(X,\sF) \to H^q(X,\sF) \to H^{q}(X \setminus Y, \sF)_{\rm an} \to \ldots $$
in the analytic case, 
see \cite{Har70}, \cite{KP90}.

\end{definition}

\begin{remark} \label{remarkmoderategrowth} We may extend the previous definition of $H^q_{[Y]}(X,\sF) $ to the case of compact complex manifolds as in \cite[Sect.1]{KP90}. 
In this case $H^{q}(X \setminus Y, \sF)_{\rm alg} $ has to be substituted by a cohomology group of moderate growth, $H^{q}(X \setminus Y, \sF)_{\rm mod} $. 
In particular, $H^{0}(X \setminus Y, \sO_{X \setminus Y})_{\rm mod}$ is the space of  those holomorphic functions on $X \setminus Y$ extending to meromorphic functions on $X$.  
\end{remark}

\begin{definition} Let $X$ be a compact complex manifold of dimension $n$, and let $\sL$ be a holomorphic line bundle on $X$.
\begin{enumerate}
\item The bundle $\sL$ is $q$-positive if there exists a metric on $\sL$ whose curvature has at least $n-q$ positive eigenvalues. 
\item Suppose that $X$ is projective. Then $\sL$ is $q$-ample if the following holds. Given any coherent sheaf $\sF$ on $X$, there is a number $k_0$, depending on $\sF$, such that 
$$ H^j(X,\sF \otimes \sL^{\otimes k}) = 0 $$
for $k \geq k_0$ and $j \geq q+1$. 
\end{enumerate} 
\end{definition}

It was conjectured ("the converse Andreotti-Grauert conjecture") that on a smooth projective $n$-fold, $(n-1)$-ample line bundles are $(n-1)$-positive, \cite{DPS96}. 
This is now established by Yang \cite{Yan19}.
In dimension $2$ this has previously been proved by Matsumura \cite[Thm.1.3]{Mat13}.  Notice that there are counterexamples to the converse Andreotti-Grauert theorem  in case $ \frac{n}{2}-1 < q < n-1$,
see \cite[Thm.10.3]{Ott12}. Of course, in case $q = 0$ the converse Andreotti-Grauert theorem is a classical result of Kodaira.

\begin{proposition} \label{prop:alg-an} 
Let $X$ be a compact complex manifold of dimension $n$, and let $Y \subset X$ be a smooth hypersurface. Assume that the conormal 
bundle $N^*_{Y/X}$ is $(n-2)$-positive. Let $\sF$ be a locally free sheaf on $X$. Then the canonical morphism
$$ H^1_{[Y]}(X,\sF) \to H^1_Y(X,\sF) $$
is bijective. 
 Equivalently, the canonical morphism
$$ H^0(X \setminus Y,\sF)_{\rm mod} \to H^0(X \setminus Y, \sF)_{\rm an}$$
is bijective. In case $X$ is algebraic, equivalently 
$$ H^0(X \setminus Y,\sF)_{\rm alg} \to H^0(X \setminus Y, \sF)_{\rm an}$$
is bijective. 

\end{proposition} 

\begin{proof} This is a special case of \cite[Thm. 2.5]{KP90} (note that $N^*_{Y/X} $ is $(n-2)$-positive if and only if $N_{Y/X}$ is $(n-1)$-convex in the notation of 
\cite{Ko82}, \cite{Sch73}). 
\end{proof}

\begin{remark*} {\rm 
Proposition \ref{prop:alg-an} remains true for singular hypersurfaces $Y$,
if one uses the notion of $(n-1)$-convexity of the normal bundle $N_{Y/X}$, see \cite[2.8]{Ko82}.
}
\end{remark*}

\section{General Stein complements} 

\subsection{General results}

We first recall 

\begin{definition} \label{definition:Stein} \cite[IV, \S 1, Defn.1]{GR79} 
Let $X$ be a complex space. We say that $X$ is Stein if for every coherent analytic sheaf $\sF$ on $X$
one has
$$
H^q(X, \sF) = 0 \qquad \forall \ q \geq 1
$$
\end{definition}

It is well-known \cite[V, \S2, Thm.3]{GR79} that  a complex space $X$ is Stein if and only if $X$
is holomorphically convex and does not contain any compact analytic subspaces of positive dimension.
We will use the holomorphic convexity of a Stein space via the following characterisation

\begin{theorem} \cite[IV, \S 2, Thm.4, Thm.12]{GR79} 
A complex space $X$ is holomorphically convex if and only if for every discrete infinite subset $D\subset X$
there exists a holomorphic function $f \in \sO_X(X)$ such that $f|_D$ is unbounded.
\end{theorem}

\begin{lemma} \label{lem:top} \cite[p.156]{GR79} 
Let $X$ be a compact complex manifold, and let $Y \subset X$ be a compact hypersurface such that $X \setminus Y$ is Stein. 
Then the restriction maps
$$ H^q(X,\mathbb Z) \to H^q(Y,\mathbb Z)$$
are isomorphisms for $q < n-1$.
In particular, if $ n \geq 4,$ and $Y$ is normal, then the restriction map 
$$ N^1(X) \to N^1(Y)$$
is an isomorphism. 
\end{lemma}

In dimension two, Problem \ref{P1}  has clearly a positive solution:

\begin{proposition} \label{prop:one}
Let $X$ be a smooth complex surface, and let $Y \subset X$ be an irreducible compact curve such that $X \setminus Y $ is holomorphically convex.  
Then the normal bundle $N_{Y/X}$ is nef.  
\end{proposition} 

\begin{proof} If $Y^2 < 0$, then by Grauert's criterion $Y$ can be blown down
by a bimeromorphic morphism $X \rightarrow X'$ onto a normal surface $X'$. Hence all holomorphic functions on $X \setminus Y$ would extend to $X$ by Riemann's extension theorem 
on $X'$. Thus  $X \setminus Y $ is not holomorphically convex.  
\end{proof} 

\begin{remark}  \cite[Thm.5.3]{KP90} {\rm 
Let $X$ be a relatively minimal projective surface, and let $Y \subset X$ be an elliptic curve such that $X \setminus Y$ is Stein and $N_{Y/X}$ is not ample. 
Then $X \simeq \PN(V)$, where $V$ is the non-split 
extension 
$$ 0 \to \sO_C \to V \to \sO_C \to 0 $$
over an elliptic curve 
and $Y = \PN(\sO_C)$. 
}
\end{remark}

Proposition \ref{prop:one} yields the following 

\begin{lemma} \label{lem1}   Let $X$ be a projective manifold of dimension $n$, and let $Y \subset X$ be an irreducible hypersurface such that $X \setminus Y$ is holomorphically convex. Then 
$$
Y^2 \cdot H_1 \cdot \ldots \cdot H_{n-2} \geq 0
$$
for all nef divisors $H_j$ on $X$. 
\end{lemma} 

\begin{proof} 
We may assume without loss of generality that the divisors $H_j$ are very ample and general. Then the surface $S:= H_1 \cap \ldots \cap H_{n-2}$
is smooth, and $C = S \cap Y$ is an  irreducible curve.
By Proposition \ref{prop:one}, one has
$$ (C^2)_S = Y^2 \cdot H_1  \cdot \ldots \cdot H_{n-2} \geq 0.$$
\end{proof}

\begin{corollary} \label{cor1} In the situation of Problem \ref{P1}, if the restriction map of the nef cones
$$ {\rm Nef}(X) \to {\rm Nef}(Y) $$
surjective, then $N_{Y/X}$ is generically nef. 
\end{corollary}

We can now prove our first main result:

\begin{proof}[Proof of Theorem \ref{thm:mainone}] 
We argue by contradiction and assume that $N_{Y/X}$ is not pseudoeffective. 
By \cite[Thm.1.5]{Yan19} this is equivalent to assuming that $N^*_{Y/X}$ is $(\dim Y-1)$-positive, hence $(n-2)$-positive. 

Since $X \setminus Y$ is Stein,  so we can find holomorphic functions $f_1, \ldots, f_N$ on $X \setminus Y$ 
yielding a closed embedding $X \setminus Y \hookrightarrow \mathbb C^N$. 
By Proposition \ref{prop:alg-an} the functions $f_j$ have moderate growth, so they extend
to $X$ (cf. Remark \ref{remarkmoderategrowth}). Hence
$X$ is at least Moishezon. 
Since $X$ is assumed to be K\"ahler, it is projective by Moishezon's theorem. 
Thus by GAGA (or by the second part of Proposition  \ref{prop:alg-an})  the functions 
$f_j$ are algebraic and the embedding is
actually an algebraic embedding $X \setminus Y \hookrightarrow \mathbb A^N$. Hence $X \setminus Y$ is affine and $N_{Y/X}$ is big by \cite[Prop.4.2]{GW20}. 
Thus we have reached a contradiction.
\end{proof}

\subsection{Examples}

\begin{setup} \label{ext}  {\rm Let $B$ be a projective manifold, and let $\sF$ be a vector bundle on $B$. Fix a cohomology class $0 \ne \zeta \in H^1(B,\sF^*)$.
Then $\zeta $ defines a non-split extension 
$$ 
0 \to \sO_B \to V \to \sF \to 0.
$$
We set
$$
X = \mathbb P(V); \ Y = \mathbb P(\sF).
$$
and denote by $\holom{\pi}{\PP(V)}{B}$ the natural map.
We set further
$ \zeta_X = c_1(\sO_{\mathbb P(V)}(1))$ and
$ \zeta_Y = c_1(\sO_{\mathbb P(\sF)}(1))$
}
\end{setup} 

\begin{lemma}  \label{lem:curves} 
In the situation of Setup \ref{ext}, assume furthermore that 
$\dim B = 1$. Then  $X \setminus Y $ does not contain compact subvarieties of positive dimension. 
\end{lemma}

\begin{proof} 
Since $X$ is projective it suffices to show that $X \setminus Y $ does not contain compact irreducible curves. Arguing by contradiction, 
let $\holom{f}{C}{X}$ be a morphism from a smooth compact curve such that $f(C) \subset X \setminus Y$. Then the composition $\pi \circ f$ is a finite map.
Let $\tau \in H^0(B, \sF^* \otimes \omega_B)$ be the Serre dual of $\zeta$. 
Then 
$$ 
0 \ne f^*(\tau ) \in H^0(C, f^* \sF^* \otimes \omega_B) \subset H^0(C, f^* \sF^* \otimes \omega_C).
$$
Since the Serre dual of $f^*(\tau)$ is $f^*(\zeta)$, we obtain
$$ 
0 \ne f^*(\zeta) \in H^1(C,f^*(\sF)). 
$$ 
Thus, up to making a base change, 
we may assume that $C$ is a section of $X \to B$ that is disjoint from $Y=\PP(\sF)$. 
Hence the exact sequence defined by $\zeta$ splits, a contradiction. 
\end{proof} 

We can now given the example that proves Proposition \ref{prop:goodman}:

\begin{example} \label{example1}{\rm 
In the situation of Setup \ref{ext}, let $B = \mathbb P_1$ and $\sF = \sO_{\PP^1}(2) \oplus \sO_{\PP^1}(-1)$. Since the extension class $\zeta$ is not zero, one 
has
$$
V \simeq \sO_{\PP^1}(1)^{\oplus 2} \oplus \sO_{\PP^1}(-1).
$$
Let $C_0 \subset Y$ be the section corresponding to the quotient $\sF \rightarrow \sO_{\PP^1}(-1)$. Then the fibrewise projection from $C_0 \subset X$ defines a rational map
$$
\merom{\psi}{X}{\PP(\sO_{\PP^1}(1)^{\oplus 2})}.
$$
More precisely, let $\pi: \hat X \to X$ be the blow-up of $X$ along $C_0$ and denote by $E$ the exceptional divisor.
Then $\hat X$ has a structure of $\PP^1$-bundle $\holom{\varphi}{\hat X}{\PP(\sO_{\PP^1}(1)^{\oplus 2})}$ such that the exceptional divisor $E$ is a section.
Moreover if $\zeta_{Q}$ is the tautological class on $\PP(\sO_{\PP^1}(1)^{\oplus 2})$,
one has
$$
\pi^*(\zeta_X) - E = \varphi^* \zeta_Q.
$$
Now observe that the strict transform $\hat Y \simeq Y$ is an element of the linear
system $|\pi^*(\zeta_X) - E|$, since $Y \in |\zeta_X|$ contains the curve $C_0$.
The isomorphism $\PP(\sO_{\PP^1}(1)^{\oplus 2}) \simeq \PP^1 \times \PP^1$
identifies the class $\zeta_Q$ to $\sO_{\PP^1 \times \PP^1}(1,1)$,
so we see that $\hat Y = \varphi^* \ell$ with $\ell$ a smooth conic.

One has $\hat X \simeq \PP(\sO_{\PP^1 \times \PP^1} \oplus \sO_{\PP^1 \times \PP^1}(-2,-1))$,
and $E$ corresponds to the quotient $\PP(\sO_{\PP^1 \times \PP^1}(-2,-1))$. Thus $\hat X \setminus E$ is the total space of the line bundle $\sO_{\PP^1 \times \PP^1}(-2,-1)$ and
$$
\hat X \setminus (\hat Y \cup E) \to (\PP^1 \times \PP^1) \setminus \ell
$$
is a line bundle over the affine surface $Q_0:= (\PP^1 \times \PP^1) \setminus \ell$. 
Hence 
$$ 
X \setminus Y \simeq \hat X \setminus (\hat Y \cup E) 
$$
is affine: if $s_1, \ldots, s_n$ are algebraic sections that generate the line bundle on
$Q_0$, they define an algebraic embedding of the total space into $Q_0 \times \C^n$.\\
But the normal bundle $N_{Y/X} = \zeta_Y$ is not nef, since $\zeta_Y \cdot C_0=-1$. 
}
\end{example} 

\begin{example} \label{example2} {\rm  Let $B$ be an elliptic curve and $\sL$ a line bundle of negative degree. 
Let 
$$ 0 \to \sO_B \to \sF \to \sO_B \to 0$$
be the non-split extension. 
Note that $\PP(\sF) \setminus \PP(\sO_B) \simeq \mathbb C^* \times \mathbb C^*$ is Stein (but not affine);  this is Serre's famous example \cite[\S 6.3]{Har70},\cite[\S 7]{Nee88}. 
Note next that $\PP(\sO_B \oplus \sL) \setminus \PP(\sL)$ is not Stein, since $\PP(\sL)$ has negative normal bundle. 
However 
$$ \PP(\sF \oplus \sL) \setminus \PP(\sO_B \oplus \sL)$$
is Stein. 
In fact, we have a projection
$$ \PP(\sF \oplus \sL) \setminus \PP(\sL) \to \PP(\sF)$$
which restricts to a morphism
$$ \Phi: \PP(\sF \oplus \sL) \setminus \PP(\sO_B \oplus \sL) \to \PP(\sF) \setminus \PP(\sO_B) \simeq \mathbb C^* \times \mathbb C^*.$$
Since $\Phi$ is an affine $\mathbb C$-bundle over $\mathbb C^* \times \mathbb C^*$, the complement
$ \PP(\sF \oplus \sL) \setminus \PP(\sO_B \oplus \sL)$ is Stein by a result of Mok \cite{Mok81}.
}
\end{example}

\subsection{Criteria for affineness}

We next give criteria when $X \setminus Y$ is affine. This will be used later in the case of canonical extensions. 

\begin{proposition} \label{prop:semiample} 
Let $X$ be a normal compact analytic space, and let $Y \subset X$ be an irreducible hypersurface
such that some multiple $mY$ is Cartier and the line bundle $\sO_X(mY)$ is semiample. 
Assume that $X \setminus Y$ does not contain any positive-dimensional compact subvarieties.

Then $X \setminus Y$ is affine. Moreover the divisor $Y$ is big and some multiple defines a
bimeromorphic map $\varphi: X \to Z$, whose exceptional locus is strictly contained in $Y$
and there exists a $\mathbb Q$-Cartier Weil divisor 
$Y' \subset  Z$ such that $Y = \varphi^{-1}(Y')$.
\end{proposition}

A similar statement is shown in the algebraic setting in \cite[Ch.II,Prop.2.2]{Har70}.

\begin{proof}  The sheaf
$\sO_X(mY)$ being generated by global sections, denote by $\varphi: X \to Z$  
the associated morphism with connected fibers, so that $\sO_X(mY) = \varphi^*(\mathcal A)$ for some very ample line bundle $\mathcal A$ on $Z$. 

In particular the divisor $mY = \varphi^*(D)$ for some effective 
Cartier divisor $D \in |A|$. 
Let now $z \in Z$ a point such that the fibre $\fibre{\varphi}{z}$ has positive dimension. 
Then $z \in D$ since otherwise $X \setminus Y$ contains a compact subvariety
of positive dimension. Thus $\varphi$ is bimeromorphic, in particular it is Moishezon by Chow's lemma
\cite[Cor.2]{Hir75}
and all the fibres of positive dimension are covered by irreducible curves.
Thus if the inclusion $\fibre{\varphi}{z} \cap Y \subset Y$ is strict, we can find irreducible curves
$C \subset \fibre{\varphi}{z}$ that meet $Y$ and such that $C \not\subset Y$. Yet this implies
$Y \cdot C>0$, contradicting the fact that $C$ is contracted by $\varphi$.
Thus we see that the exceptional locus of $\varphi$ is contained in $Y$. 

Since $mY$ is the pull-back of an ample divisor, it is not contracted by $\varphi$. Thus the exceptional locus is strictly contained in $Y$. Since $mY = \varphi^* D$ we obtain that
$D = m Y'$ where $Y' = \varphi(Y)$. Since $D$ is Cartier, the divisor $Y'$ is $\Q$-Cartier.
\end{proof} 

\begin{remark} {\rm 
If $H^1(X,\sO_X) = 0$ and the normal bundle $N_{Y/X}$ is generated by global sections,
then $Y$ is semiample and Proposition \ref{prop:semiample}
applies.

Without assuming $H^1(X,\sO_X) = 0$, this is false: in fact, let $A$ be an abelian variety and let 
$$ 0 \to \sO_A \to V \to T_A \to 0,$$
be the canonical extension of $T_A$ given by some K\"ahler class (see Section \ref{sectioncanonical}). 
Set $X = \PN(V) $ and $Y = \PP(T_A)$. Then by  \cite[Prop.2.13]{GW20}
the complement $X \setminus Y$ is isomorphic to $(\mathbb C^*)^{2n}$, hence Stein. 
Yet the normal bundle $N_{Y/X} = \zeta_{\PP(T_A)}$ is globally generated, but not big. 
} 
\end{remark}

In special situations, the semiampleness assumption in Proposition \ref{prop:semiample} can be checked. We consider right away a singular setting.

\begin{proposition} \label{prop:bignef} Let $X$ be a $\mathbb Q$-factorial projective klt variety, $Y \subset X$ an irreducible hypersurface.  Assume that $Y$ is nef and $aY-K_X$ is nef and big for some $a>0$.
If $X\setminus Y$ does not contain positive-dimensional compact subvarieties, then $X \setminus Y$ is affine. 
\end{proposition} 

\begin{proof} By the base point free theorem, the divisor $Y$ is semiample.
Now Proposition \ref{prop:semiample} applies.
\end{proof} 

Note that the assumption that $aY-K_X$ is nef and big for some $a>0$ holds if
\begin{itemize}
\item $Y$ is nef and big and $-K_X$ is nef; or
\item $Y$ is nef $-K_X$ is nef and big (i.e. $X$ is weakly Fano).
\end{itemize}

\subsection{Use of Mori theory} \label{subsectionmori}

In the situation of Problem \ref{P1}, assume that $X$ is projective and the divisor
$Y$ is not nef (equivalently, the normal bundle $N_{Y/X}$ is not nef). We will use Mori theory to analyze this case.

\begin{proposition} \label{prop:contract-1} Let $X$ be a $\mathbb Q$-factorial projective klt variety, and let $Y \subset X$ be an irreducible hypersurface such that $X \setminus Y$ is Stein. 
Suppose that there is an extremal ray $R \subset \overline{NE}(X) $ with $Y \cdot R < 0 $ and $K_X \cdot R \leq 0$ or with $Y \cdot R = 0$ and $K_X \cdot R < 0$. 
Then the contraction
$$ 
\varphi: X \to X'
$$
defined by the $(K_X + \epsilon Y)$-negative ray $R$ is small with exceptional locus properly contained in $Y$. 
Furthermore there exists a flip $$X \dasharrow X^+$$ such that $X \setminus Y \simeq X^+ \setminus Y^+$, where $Y^+$ is the strict transform of $Y $ in $X^+$. 
\end{proposition} 

\begin{proof} For small positive $\epsilon$, the pair $(X, \epsilon Y)$ is klt and $(K_X + \epsilon Y) \cdot R < 0$. The existence of $\varphi$ then follows from the contraction theorem. 
Since $Y \cdot R \leq 0$ and since $X \setminus Y$ is Stein, 
every curve contracted by $\varphi$ is contained in $Y$. Thus the exceptional
locus $E$ is contained in $Y$, and $X' \setminus \varphi(Y) \simeq X \setminus Y$
is Stein. By \cite[V,Thm.4]{GR79} this implies that $\varphi(Y) \subset X'$
has codimension one, in particular $Y$ is not contracted by $\varphi$.
Thus the contraction $\varphi$ is small.

By  \cite{BCHM10} the flip of $\varphi$ exists, and we denote by $\varphi^+: X^+ \to X'$
the induced map, which is small. Since $X' \setminus \varphi(Y)$ is $\Q$-factorial,
the image of the $\varphi^+$-exceptional locus is contained in $\varphi(Y)$.
Hence if $Y^+ \subset X^+$ is the strict transform of $Y$, then the exceptional locus 
$E^+$ is contained in $Y^+$.

\end{proof}

\begin{corollary} \label{cor_flip}  Let $X$ be a $\mathbb Q$-factorial projective klt variety, and let $Y \subset X$ be an irreducible hypersurface such that $X \setminus Y$ is Stein. 
Suppose that one of the following conditions is satisfied:
\begin{enumerate} 
\item $\dim X \leq 3$;
\item $\dim X = 4$ and $(X,\epsilon Y)$ is canonical for some small positive $\epsilon$;
\item $\dim X$ is arbitrary and any sequence of (log) flips on $X$ terminates. 
\end{enumerate} 
Then there exists a finite sequence of (log) flips 
$$ X \dasharrow \tilde X$$
with strict transform $\tilde Y \subset \tilde X$ such that the following holds.
\begin{enumerate}
\item $ \tilde X \setminus \tilde Y \simeq X \setminus Y$ is Stein ;
\item for any extremal ray  $\tilde R \in \overline{NE}(\tilde X)$ with $\tilde Y \cdot \tilde R < 0$, we have $K_{\tilde X} \cdot \tilde R  > 0$;
\item for any extremal ray  $\tilde R \in \overline{NE}(\tilde X)$ with $\tilde Y \cdot \tilde R = 0$, we have $K_{\tilde X} \cdot \tilde R \geq 0$. 
\end{enumerate}
\end{corollary} 

\begin{proof} The statement follows from the repeated application of Proposition \ref{prop:contract-1}, having in mind the termination of the procedure in dimension four due to our extra assumption,
\cite{Kaw92}, \cite{Fuj04}, \cite{Fuj05}.
\end{proof}

These statements can be made more explicit in a number of special cases:

\begin{theorem}  \label{cor:Fano} Let  $X$ be a Gorenstein Fano threefold or a $\mathbb Q$-Fano variety of any dimension without a small contraction. 
If $Y \subset X$  is an irreducible hypersurface such that  $X \setminus Y$ is Stein, then $Y$ is ample.  
In particular, $X \setminus Y$ is affine. 
\end{theorem} 

\begin{proof} By \cite{Cut88} we know that Gorenstein Fano threefolds do not support small contractions of extremal rays. By Proposition \ref{prop:contract-1} we obtain that $Y \cdot R > 0$ for all $K_X$-negative extremal rays. Since $X$ is Fano, the cone theorem now implies that $Y$ is ample.  
\end{proof}

\begin{proposition}\label{prop:n} 
Let $X$ be a projective $\mathbb Q$-factorial klt variety with $K_X \equiv 0$. Let $Y \subset X$ be an irreducible hypersurface such that $X \setminus Y$ is Stein. 
Suppose that one of the following conditions is satisfied:
\begin{enumerate} 
\item $\dim X \leq 3$;
\item $\dim X = 4$ and $(X,\epsilon Y)$ is canonical for some small positive $\epsilon$;
\item $\dim X$ is arbitrary and $Y$ is big. 
\end{enumerate} 
Then there exists a birational rational map $f: X \dasharrow \tilde X$ with the following properties.
\begin{enumerate}
\item $f $ is a sequence of flops; the exceptional locus being contained in $Y$;
\item $\tilde X$ is a $\mathbb Q$-factorial klt variety with $K_{\tilde X} \equiv 0$;
\item let $\tilde Y$ denote the (possibly singular) strict transform of $Y$ in $\tilde X$, then $\tilde Y$ is nef.
\end{enumerate} 
Moreover, if $\dim X \leq 3$ or $Y$ is big, then $\tilde Y$ is semiample and $X \setminus Y$ is affine. 
\end{proposition}

\begin{proof}
The existence of the birational model is given by Corollary \ref{cor_flip} (note that if $Y$ is big,
we may run a $(K_X +  \epsilon Y)$-MMP by \cite{BCHM10}). 

If $\dim X=3$ it follows from log abundance \cite{KMM94, KMM94b} that $\tilde Y$ is semiample.
If $\dim X$ is arbitrary and $Y$ is big, semiampleness is guaranteed from the basepoint-free theorem.
Hence $X^+ \setminus Y^+ \simeq X \setminus Y$ is affine by Corollary \ref{prop:semiample}.
\end{proof}

In dimension three we can prove more, at least in a smooth setting.

\begin{theorem} Let $X$ be a smooth projective threefold, $Y \subset X$ a smooth hypersurface such that $X \setminus Y$ is Stein. 
Let $z \in \overline{NE}(X) $ with $K_X \cdot z < 0$. Then $Y \cdot z \geq 0$.
\end{theorem} 

\begin{proof}  Suppose to the contrary that $Y \cdot z < 0$. Then there exists an irreducible curve $C$ such that $K_X \cdot C < 0$ and $Y \cdot C < 0$. 
Since $C \subset Y$, the adjunction formula yields $K_Y \cdot C \leq -2$. Since $Y \cdot C < 0$ and $N_{Y/X}$ is pseudoeffective by Theorem \ref{thm:mainone},
the curve $C$ is in the negative part of the Zariski decomposition of the divisor $N_{Y/X}$.
In particular we obtain   $\deg N_{C/Y} < 0$. Hence the adjunction formula on $Y$ yields 
$\deg \omega_Y \leq K_Y \cdot C < -2$. Yet it is well-known that $\deg \omega_Y \geq -2$
for any irreducible curve. 
\end{proof}

 \section{The canonical extension} 
 \label{sectioncanonical}
 
 We will now consider the special case of the canonical extension of $\PP(T_M)$. In this case the geometric interpretation of the normal bundle allows to prove stronger results.

 \subsection{The smooth case} 
 
\begin{setup} \label{Setup1} 
Let $M$ be a compact K\"ahler manifold, and let $\alpha$ be a K\"ahler class on $M$. 
Then  the class $\alpha \in H^1(M,\Omega_M) \simeq {\rm Ext}^1(\sO_M,T_M) $ defines a  non-split extension 
\begin{equation} \label{EM} 
0 \to \sO_M \to V^{\alpha} \to T_M \to 0. 
\end{equation}
Denote by $\holom{\pi}{\PP(V^\alpha)}{M}$ and $\holom{\pi_M}{\PP(T_M)}{M}$
the projectivisations, and by $\zeta_V$ (resp. $\zeta_M$) the tautological class on $\PN(V^\alpha)$ (resp. $\PP(T_M)$). 

The exact sequence \eqref{EM} determines an embedding
$$
\mathbb P(T_M) \subset \mathbb P(V^{\alpha})
$$
such that $[\mathbb P(T_M)] = \zeta_V$ and we denote
$$
Z_M^{\alpha} := \mathbb P(V^{\alpha}) \setminus \mathbb P(T_M).
$$

When $\alpha$ is fixed, we will simply write $V = V^{\alpha} $ and $Z_M = Z_M^{\alpha}.$ 
\end{setup}

\begin{proof}[Proof of Theorem \ref{prop:nefFano}]
 Since $T_M$ is nef and big, $\PNM$ is a nef and big divisor on $\PN(V)$. Moreover, since $\det V = -K_M$, the anticanonical class 
$-K_{\PN(V)} $ is nef and big. Hence Proposition \ref{prop:bignef} applies.
\end{proof} 

Theorem  \ref{prop:nefFano} applies in particular to rational-homogeneous manifolds, reproving a part of \cite[3.4]{GW20}. It is however conjectured, \cite{CP91}, 
$T_M$ is big and nef if and only $M$ is rational homogeneous. 

In the general  homogenenous case we may state

\begin{corollary} Let $M$ be a homogeneous compact K\"ahler manifold. Then $Z_M$ is Stein for any K\"ahler class $\alpha$ on $M$. 
\end{corollary} 

\begin{proof} By the theorem of Borel-Remmert, $M \simeq N \times T$ with $N$  rational homogeneous, in particular Fano, and $T$ a torus. 
Hence $Z_M \simeq Z_N \times Z_A$ is Stein by \cite[2.10, 3.4]{GW20}. 
\end{proof}

\subsection{The singular case and contractions of extremal rays}

In order to understand the birational geometry of canonical extensions we have
to generalise the setup to mildly singular spaces.

Let $M$ be a normal $\Q$-factorial complex space with klt singularities, and let 
$\alpha \in N^1(M)$ be a $(1,1)$-class with local potentials (cf. \cite{HP16} for definitions). By \cite{GS21} we can associate to $\alpha$ a cohomology class in $H^1(M, \Omega_M)$ which we also denote by $\alpha$.
Since we have $H^1(M, \Omega_M) \simeq \mbox{Ext}^1(\sO_M, \Omega_M)$ the class $\alpha$
defines an extension 
$$
0 \rightarrow \Omega_M \rightarrow W_M^{\alpha} \rightarrow \sO_M \rightarrow 0
$$
that is locally splittable. Thus dualising yields a  locally splittable exact sequence 
\begin{equation} \label{extensionalpha}
0 \rightarrow \sO_M \rightarrow V_M^\alpha \rightarrow T_M \rightarrow 0.
\end{equation}
The surjection $V_M^\alpha \rightarrow T_M$
induces an inclusion $\PP(T_M) \subset \PP(V_M^\alpha)$. If $M$ is singular, the complex spaces $\PP(T_M)$ and $\PP(V_M^\alpha)$ might be reducible, so we refine the construction:

\begin{definition} Let $M$ be a normal complex space and $\sE$ a non-zero torsion free sheaf on $M$. We denote by $\PP'(\sE)$ the unique irreducible component of $\holom{\pi}{\PP(\sE)}{M}$ that dominates $M$.
\end{definition}

\begin{definition}
Let $M$ be a normal $\Q$-factorial complex space with klt singularities, and let 
$\alpha \in N^1(M)$ be a (1,1)-class. Set $Z_M^\alpha:=  \PP'(V_M^\alpha) \setminus \PP'(T_M)$. We call  
$$
\holom{\pi}{Z_M^\alpha}{M}
$$
the canonical extension associated to the $(1,1)$-class $\alpha$.
\end{definition}

Note that contrary to the Setup \ref{Setup1}, for technical reasons we do not assume that the class $\alpha$ is K\"ahler.
The goal of this section is to show the following result:

\begin{theorem} \label{theorem-contractions} 
Let $M$ be a normal $\Q$-factorial compact K\"ahler space with klt singularities.
Let $\holom{\varphi}{M}{N}$  be a bimeromorphic morphism contracting an irreducible divisor $E$ onto an analytic set $B$. Assume that $N$ is klt, the higher direct images $R^j \varphi_* \sO_M$ vanish and that $\rho(M/N)=1$.
Assume furthermore one of the following
\begin{itemize}
\item $B$ is not contained in $N_{\sing}$;
\item $\dim M \geq 3$ and $B$ is an isolated hypersurface singularity of $N$;  
\item $B$ is an irreducible component of $N_{\sing}$ and $N$ has quotient singularities in a general point of $B$.
\end{itemize}
Then for any K\"ahler class $\alpha$, the canonical extension $Z_\alpha$ is not Stein.
\end{theorem} 

\begin{remark*}
The technical assumption that $N$ is klt, that the higher direct images $R^j \varphi_* \sO_M$ vanish and that $\rho(M/N)=1$ is of course satisfied if $\varphi$ is the contraction of an extremal ray in $\NEM$
such that $-K_M$ is $\varphi$-nef.
\end{remark*}

\begin{corollary} \label{cor:minimal}
Let $M$ be a smooth compact K\"ahler surface. Assume that $M$ contains a smooth rational curve $C \simeq \PP^1$ such that $C^2<0$.
Then for any K\"ahler class $\alpha$, the canonical extension $Z_M$ is not Stein.
\end{corollary}

\begin{proof}
By Grauert's criterion the curve $C$ can be contracted onto a point by a bimeromorphic morphism $\holom{\varphi}{M}{N}$.
Since $C \simeq \PP^1$, one has $R^1 \varphi_* \sO_M=0$ and the complex surface $N$ is klt \cite[Ex.3.8]{Kol97}.
Hence $N$ has at most a quotient singularity in $\varphi(C)$ \cite[Prop.4.18]{KM98} and we may apply Theorem \ref{theorem-contractions}.
\end{proof}

\begin{setup} \label{Setup2}{\rm
In the situation of Theorem \ref{theorem-contractions}, fix a K\"ahler class $\alpha$ on $M$.
Then there exists a unique $\lambda \in \R$
such that $\alpha - \lambda E$ is $\varphi$-numerically trivial, so by \cite[Lemma 3.3]{HP16}
there exists a unique class $\beta \in N^1(N)$ such that $\alpha - \lambda E = \varphi^* \beta$. 

Denote by $\holom{\pi}{Z_M^\alpha}{M}$ the canonical extension associated to $\alpha$, and by 
$\holom{\psi}{Z_N^{\beta}}{N}$ the canonical extension associated to the class $\beta$.
}
\end{setup}

\begin{lemma} \label{lemmaprepareopen}
In the situation of Setup \ref{Setup2}, assume that $Z_M^\alpha$ is Stein. Then $Z_N^{\beta} \setminus \fibre{\psi}{B}$ is Stein.
\end{lemma}

\begin{proof}
Since $M \setminus E \simeq N \setminus B$ and the restriction of the class of $E$ to $M \setminus E$ is trivial, it follows from the construction of $\beta$ that
$$
Z_M^\alpha\setminus \fibre{\pi}{E} \simeq Z_N^{\beta} \setminus \fibre{\psi}{B}.
$$
Since $E$ is a prime divisor and $M$ is $\Q$-factorial, there exists a $m \in \N$
such that $mE$ is Cartier. In particular $\pi^* mE$ is Cartier, hence by
\cite[V, \S 1, Thm.5]{GR79} the complex space 
$Z_M^\alpha\setminus \fibre{\pi}{E}$
is Stein. 
\end{proof}

We start by proving a rather technical lemma which will be useful for the proof of Theorem \ref{theorem-contractions}.

\begin{lemma} \label{lemmaextensionproperty}
Let $N$ be a normal $\Q$-factorial K\"ahler space with klt singularities.
Assume that there exists an irreducible component $B \subset N_{\sing}$ such that $N$ has quotient singularities at a general point of $B$.
Let 
$$
(*) \qquad 0 \rightarrow \sO_N \rightarrow V \rightarrow Q \rightarrow 0
$$
be an extension of torsion-free coherent sheaves that are locally free in $N_{\nons}$. 

We say that $(*)$ has the extension property near $B$
if the following holds: let $y \in B$ be a general point, and let $N' \subset N$ be an analytic neighbourhood of $y$ such that there exists a quasi-\'etale cover $\holom{p}{\tilde N}{N'}$ such that $\tilde N$ is smooth.
Denote by $R \subset \tilde N$ the preimage of the singular locus $B \cap N'$. Then the exact sequence
$$
0 \rightarrow \sO_{\tilde N \setminus R} \rightarrow p^* (V \otimes \sO_{N' \setminus B})
\rightarrow p^* (Q \otimes \sO_{N' \setminus B}) \rightarrow 0
$$
extends to an exact sequence of vector bundles
$$
0 \rightarrow \sO_{\tilde N} \rightarrow V_{\tilde N} \rightarrow Q_{\tilde N} \rightarrow 0.
$$
Assume now that the exact sequence $(*)$ has the extension property near $B$, and
consider the complex space $$Z: = \PP'(V) \setminus \PP'(Q).$$
Then $Z \setminus \fibre{\psi}{B}$ is not holomorphically convex, where $\psi: Z \rightarrow N$ is the natural map.
\end{lemma}

\begin{proof} We argue by contradiction and assume that $Z \setminus \fibre{\psi}{B}$
is holomorphically convex.
Fix a general point $y \in B$. Since $(*)$ has the extension property near $B$, we can consider
the quasi-\'etale cover $\holom{p}{\tilde N}{N'}$ and
the exact sequence of vector bundles
$$
0 \rightarrow \sO_{\tilde N} \rightarrow V_{\tilde N} \rightarrow Q_{\tilde N} \rightarrow 0.
$$
appearing in the definition.  The difference $Z_{\tilde N} := \PP(V_{\tilde N}) \setminus \PP(Q_{\tilde N})$ is an affine bundle $\holom{\tilde \psi}{Z_{\tilde N}}{\tilde N}$.
Choose a discrete sequence $(z_n)_{n \in \N}$ of points in $Z_{\tilde N} \setminus \fibre{\tilde \psi}{R}$
converging to a point $z_\infty \in \fibre{\tilde \psi}{R}$.
Since $p$ is \'etale in the complement of $B \cap N'$, we have a natural map
$$
\holom{\tau}{Z_{\tilde N} \setminus \fibre{\tilde \psi}{R}}{Z \setminus \fibre{\psi}{B}}
$$
that is finite onto its image.
The sequence $(\tau(z_n))_{n \in \N}$ is discrete in $Z \setminus \fibre{\psi}{B}$, so 
by \cite[IV, \S 2, Thm.12]{GR79} there exists a holomorphic function $f$ on $Z \setminus \fibre{\psi}{B}$ such that 
$\lim_{n \in \N}\vert f(z_n) \vert =\infty$. Thus $\tau^* f$ is a holomorphic function on $Z_{\tilde N} \setminus \fibre{\tilde \psi}{R}$ 
that is unbounded near $z_\infty$. 
Yet, since $Z_{\tilde N} \rightarrow \tilde N$ is equidimensional and $R \subset \tilde N$ has codimension at least two, the preimage $\fibre{\tilde \psi}{R} \subset Z_{\tilde N}$ has codimension at least two. Thus $\tau^* f$ extends to $Z_{\tilde N}$ by Hartog's theorem, in particular it is bounded near $z_\infty$, a contradiction.
\end{proof}

\begin{proof}[Proof of Theorem \ref{theorem-contractions}]
By Lemma \ref{lemmaprepareopen} it is sufficient to show that $Z_\beta \setminus \fibre{\psi}{B}$ is not Stein. 

If $N$ is smooth in the general point of $B$, the fibration $Z_\beta \rightarrow N$ is smooth in an analytic neighbourhood of a general point $y \in B$.
Since $B \subset N$ has codimension at least two, its preimage $\fibre{\psi}{B}$ has codimension at least two in any point mapping onto $y$.
By \cite[V,Thm.4]{GR79} this implies that $Z_\beta \setminus \fibre{\psi}{B}$ is not Stein. 

If $E$ is contracted onto an isolated hypersurface singularity, observe first that since the
extension defined by $\beta$ is locally splittable, the fibre of $\fibre{\psi}{B}$ is not empty.
Indeed locally near the point $B \in N$ we have $V_{\beta} \simeq \sO_N \oplus T_N$, so the inclusion
$\PP'(T_N) \subset \PP'(V_{\beta})$ is strict over $B$. The fibres of $\PP(V_{\beta}) \rightarrow N$
are linear projective spaces \cite[\S 2]{AT82}, and by Lemma \ref{lem:tangent} below,
the fibre over the point $B$ has dimension $\dim M+1$. Thus $\fibre{\psi}{B}$
has dimension $\dim M+1$. Since $\dim M \geq 3$, we see that
$\fibre{\psi}{B} \subset Z_\beta$ has codimension at least two, so we conclude again by
\cite[V,Thm.4]{GR79}.

Finally in the third case, by Lemma \ref{lemmaextensionproperty} it is sufficient to show that
that the exact sequence 
$$
(*) \qquad 0 \rightarrow \sO_N \rightarrow V_{\beta} \rightarrow T_N \rightarrow 0
$$
has the extension property. For a general point $y \in B$, consider a quasi-\'etale map
as in Lemma \ref{lemmaextensionproperty} and note that we can choose $N' \simeq \tilde N/G$
where $G$ is a finite group acting on $\tilde N$.
Since $H^1(N,\Omega_N) \simeq H^1(\tilde N,\Omega_{\tilde N})^G$, the extension class $\beta|_N$ corresponds to a class $\tilde \beta \in H^1(\tilde N,\Omega_{\tilde N})^G$. Then
$\tilde \beta$ defines an exact sequence of vector bundles
$$
0 \rightarrow \sO_{\tilde N} \rightarrow V_{\tilde \beta} \rightarrow T_{\tilde N} \rightarrow 0.
$$
Since $p|_{\tilde N \setminus R}$ is \'etale, this exact sequence extends the pull-back of $(*)$.
\end{proof}

\begin{lemma} 
\label{lem:tangent} 
Let $ 0 \in N \subset \mathbb C^{n+1 }  $be a local isolated hypersurface singularity. Then
$$  \sT_{N,0} / m_0 T_{N,0}  \simeq m_0/m_0^2 \simeq \mathbb C^{n+1},$$
where $m_0 \subset \sO_{N,0}$ is the maximal ideal. 
\end{lemma} 

\begin{proof}
 Dualizing the cotangent sequence
$$ 0 \to \sO_N \simeq N^*_{N/\mathbb C^{n+1}} \to (\Omega_{\mathbb C^{n+1}})_{\vert N} \to \Omega_ N \to 0 $$
yields
$$ 0 \to \sT_N \to (\sT_{\mathbb C^{n+1}})_{\vert N} \to \sJ \to 0,$$
where $\sJ $ is the ideal sheaf of $\{0\} $ in $N$, so that $\sJ_{0} = m_0$. 
Tensoring with $\sO_N/\sJ$, i.e., restricting to $\{0\}$, gives an exact sequence
$$ {\rm Tor}^1((\sT_{\mathbb C^{n+1}})_{\vert N}, \sO_N/\sJ) \to {\rm Tor}^1(\sJ, \sO_N/\sJ) \to \sT_N/\sJ \cdot \sT_N \to (\sT_{\mathbb C^{n+1}})_{\vert N} {\buildrel{\mu} \over { \to}} \sJ /\sJ^2 \to 0. $$ 
The first Tor group vanishes, since $\sT_{\mathbb C^{n+1}}$ is locally free. 
Since $\mu$ is an isomorphism, it follows
$$  {\rm Tor}^1(\sJ, \sO_N/\sJ) \simeq \sT_N/\sJ \cdot \sT_N.$$ 
In order to compute this Tor group, we restrict - in the same way as before -  the ideal sheaf sequence for $\sJ \subset \sO_N$ to the point $0$ to obtain an isomorphism
$$  {\rm Tor}^1(\sJ, \sO_N/\sJ) \simeq \sJ/\sJ^2,$$
which gives our claim.
\end{proof}

\begin{proof}[Proof of Corollary \ref{corollary-Mori-low-dimension} ] Arguing by contradiction we assume that $M$ supports a birational  Mori contraction $\varphi: M \to N$.
Since Mori's arguments \cite{Mor82} are local near the exceptional locus, we can apply the results in the K\"ahler case: the contraction is divisorial, and if $N$ is not smooth the divisor $E$ is contracted onto a point. Moreover, in the latter case the point $\varphi(E)$ is a quotient or hypersurface singularity.  In all of these cases Theorem \ref{theorem-contractions} 
gives a contradiction.

Since $M$ admits no birational contraction, it is a minimal model or Mori fibre space (\cite{Mor82} or 
\cite{HP16, HoPe2} in the K\"ahler case).
\end{proof} 

\begin{proof}[Proof of Proposition \ref{proposition-fourfold-flip}]
Arguing by contradiction we assume that $M$ admits a contraction of a $K_M$-negative extremal ray $\holom{\varphi}{M}{N}$ that is small. Then by \cite[Thm.1.1]{Kaw89}\footnote{The statement is for projective fourfolds, but the proof is local around the image of the exceptional locus, so it also works in the K\"ahler case.} we know that the exceptional locus is a disjoint union of
two-dimensional projective spaces $\PP^2$ such that $N_{\PP^2/M} \simeq \sO_{\PP^2}(-1)^{\oplus 2}$. Thus the flip can be constructed as follows: denote by $\holom{\psi}{\Gamma}{M}$ the blowup of the exceptional locus. Then each exceptional divisor is isomorphic to $\PP^2 \times \PP^1$ and one can contract them onto $\PP^1$ by a bimeromophic morphism
$\holom{\psi^+}{\Gamma}{M^+}$. In particular $M^+$ is also smooth, and we denote by
$\holom{\varphi^+}{M^+}{N}$ the $K_{M^+}$-positive contraction. 

Let now $\alpha$ be a K\"ahler class on $M$ such that $\holom{\pi}{Z_M^\alpha}{M}$ is Stein. We set
$\alpha^+:=(\psi^+)_* \psi^* \alpha \in H^1(M^+, \Omega_{M^+})$ and denote by
$\holom{\pi^+}{Z_{M^+}^{\alpha^+}}{M^+}$ the extension defined by this class. Note that, as in Setup
\ref{Setup2}, we do not assume that $\alpha^+$ is a K\"ahler class, we will only use that $Z_{M^+}^{\alpha^+}$ is an affine bundle defined by the $(1,1)$-class $\alpha^+$.
Since the restriction of the classes $\alpha$ and $\alpha^+$ to
$$
M \setminus \mbox{Exc}(\varphi) \simeq M^+ \setminus \mbox{Exc}(\varphi^+) 
$$
coincide, we have a natural isomorphism of affine bundles
$$
\tau_0 : Z_{M^+}^{\alpha^+} \setminus \fibre{(\pi^+)}{\mbox{Exc}(\varphi^+)} \rightarrow 
Z_{M}^{\alpha} \setminus \fibre{\pi}{\mbox{Exc}(\varphi)}. 
$$
Since  $\fibre{(\pi^+)}{\mbox{Exc}(\varphi^+)} \subset Z_{M^+}^{\alpha^+}$ has codimension more than two and $Z_{M}^{\alpha}$ is Stein, we know by \cite[Thm.2]{AS60} that $\tau_0$ extends
to a holomorphic map
$$
\tau : Z_{M^+}^{\alpha^+} \rightarrow 
Z_{M}^{\alpha}. 
$$
Yet such a map can't exist: fix $u \in \mbox{Exc}(\varphi^+)$ a point, and let $U \subset M^+$
be a small analytic neighbourhood such that the affine bundle $\holom{\pi^+}{Z_{M^+}^{\alpha^+}}{M^+}$ admits a section $s: U \rightarrow Z_{M^+}^{\alpha^+}$. Since $\tau_0$ is an isomorphism of affine bundles, we see that the restriction of
$$
\pi \circ \tau \circ s : U \rightarrow M
$$
to $U \setminus \mbox{Exc}(\varphi^+)$ is the identity. Thus $\pi \circ \tau \circ s$ extends the identity over the exceptional locus, hence the natural map
$M^+ \setminus \mbox{Exc}(\varphi^+) \rightarrow M \setminus \mbox{Exc}(\varphi)$
extends to a holomorphic map $M^+ \rightarrow M$, a contradiction.  
\end{proof}

\subsection{Proof of Theorem \ref{theorem-not-uniruled}}

We start by proving (a slightly more precise version of) Proposition \ref{proposition-gen-ample}: 

\begin{proposition}  \label{prop:gen-ample}
Let $M$ be a projective manifold of dimension $n$ such that $T_M$ is pseudoeffective. 
If $M$ is not uniruled, there exists exists a decomposition $$T_M \simeq \sF \oplus \sG,$$ where
$\sF \neq 0$ and $\sG$ are integrable subbundles such that $c_1(\sF)=0$.
In particular, by \cite[Cor.11]{Pe11} the manifold $M$ is not of general type and $\Omega_M$ is not generically ample.

Moreover the decomposition can be chosen such that if $H$ is a polarisation on $M$, and 
$C$ a MR-general curve associated to $H$, then $\sG|_C$ is antiample.
\end{proposition}

The proof is a rather straightforward extension of \cite{Pe11}, based on \cite{LPT18}:

\begin{proof} 
Fix a polarisation $H$ and a MR-general curve $C$ associated to $H$.
Since $T_M$ is pseudoeffective, the restriction $T_M \vert C$ is pseudoeffective.
Hence $\Omega_M \vert_C$ cannot be 
ample by Lemma \ref{lemmapseffample}. In particular $\Omega_M$ is not generically ample with respect to $H$. Now we follow the proof of \cite[Prop.2.]{Pe11}: let $$\sG_C^* \subset \Omega_M|_C$$ be the maximal ample subbundle, and set
$$\sF_C^*:= \Omega_M|_C/\sG_C^*.$$ 
Since $\mu(\sG_C^*)>0$
and $\mu(\sF_C^*)=0$, we see that $\sG_C^*$ appears in the Harder-Narasimhan filtration of $\Omega_M|_C$. 
Since $C$ is MR-general, by the theorem of Mehta-Ramanathan, there exists a saturated subsheaf  $\sG^* \subset \Omega_M$ such that  $\sG_C^* = \sG^*|_C$. We set $\sF^*:= \Omega_M/\sG^*$. By \cite{BDPP13} the determinant of $\sF^*$ is pseudoeffective. Yet $$c_1(\sF^*) \cdot H^{n-1} = \mu(\sF_C^*)=0,$$ so
$c_1(\sF^*)=0$. Thus $\sF \subset T_M$ is a foliation with $c_1(\sF)=0$.
Since $K_M$ is pseudoeffective, we know by \cite[Thm.5.2]{LPT18} that $\sF$ is integrable and the inclusion $\sF \subset T_M$ splits, inducing a decomposition into 
integrable subbundles $$T_M \simeq \sF \oplus \sG.$$

Finally if for some polarisation $H'$, the restriction $\sG|_{C'}$ is not ample, then we can repeat the argument above: since $\sF|_{C'}$ is numerically trivial, we obtain a 
maximal ample subbundle $$(\sG')_{C'}^* \subsetneq   \sG|_{C'}^*  \subset \Omega_M|_{C'}$$
leading to a decomposition $T_M \simeq \sF' \oplus \sG'$ with $\rk \sF < \rk \sF'$.
Since $\rk T_M$ is finite this process stops after finitely many steps.
\end{proof}

\begin{proposition}  \label{proposition-foliation-algebraic}
Let $M$ be a non-uniruled projective manifold of dimension $n$ such that $T_M$ is pseudoeffective. 
Let $T_M \simeq \sF \oplus \sG$ be the decomposition into integrable subbundles given by Proposition \ref{prop:gen-ample}. If the dual $\sF^*$ is not pseudoeffective, then $\sF$ is pseudoeffective. 
Moreover, up to replacing $M$ by a finite \'etale cover, one has
$$
M \simeq A \times B
$$
with $A$ an abelian variety of dimension $d>0$.
\end{proposition}

\begin{proof}
Since $\sF^*$ is not pseudoeffective we know by \cite[Prop.8.4]{Dru18} that the foliation $\sF$ has algebraic leaves. Thus by \cite[Thm.5.8]{LPT18} we know that, up to replacing $M$ by a finite \'etale cover, the decomposition $T_M \simeq \sF \oplus \sG$ is induced 
by a product structure $M \simeq N \times Z$ such that $\sF=T_{M/Z}$ and $\sG=T_{M/N}$.
The property of having pseudoeffective tangent bundle is preserved under \'etale cover \cite[Prop.4.4]{HP20}, so we now argue as in the proof of \cite[Thm.1.6]{HP19} that $\sF$ or $\sG$ is pseudoeffective. Since $\sG$ is generically anti-ample for a polarisation, it is not pseudoeffective. Thus $\sF$ is pseudoeffective, this shows the first statement.

For the second statement note that the restriction $\sF|_{N \times p} \simeq T_N$ is pseudoeffective and that $c_1(N) = 0$. By \cite[Thm.1.6]{HP19} this implies that the Beauville-Bogomolov decomposition of $N$ contains an abelian factor. Thus, up to replacing $N \times Z$ by an \'etale cover, we can assume that $N$ is an abelian variety.
\end{proof}

\begin{proof}[Proof of Theorem \ref{theorem-not-uniruled}]
We start by observing that if the manifold $M$ admits an \'etale cover by a product $A \times B \rightarrow M$, we know by \cite[Lemma 2.10]{GW20} 
that the canonical complex extension $Z_A$ and $Z_B$ are Stein. In this case the statement thus follows by induction on the dimension.

Consider now the decomposition $T_M \simeq \sF \oplus \sG$ given by Proposition \ref{prop:gen-ample}. If $\rk \sF=3$, then $c_1(M)=c_1(\sF)=0$ the statement follows from the
Beauville-Bogomolov decomposition and \cite[Thm.1.6]{HP19}.
If $\rk \sF < 3$ the manifold is a product after \'etale cover by \cite[Thm.D]{PT13} and \cite[Thm.1.2]{Tou08}, so the induction hypothesis applies.
\end{proof}

\section{Surfaces}  \label{sectionsurfaces}

In this section we discuss canonical extensions on surfaces. We recall the setup:
let $M$ be a compact K\"ahler surface, and $\alpha$ a K\"ahler class on $M$. 
Let 
$$
0 \rightarrow \sO_M \rightarrow V \rightarrow T_M \rightarrow 0
$$
the extension defined by $\alpha$, and set 
$$
Z_M :=  \PP(V) \setminus  \PP(T_M).
$$
As before we denote by $\holom{\pi_M}{\PP(T_M)}{M}$ and $\holom{\pi}{\PP(V)}{M}$ the projectivisations.

The following result is a first step towards Theorem \ref{thm:maintwo}.

\begin{proposition} \label{proposition-step} Let $M$ be a smooth projective surface. Assume that $Z_M$ is Stein for some K\"ahler class. Then $M$ is either an \'etale quotient of a torus, the projective plane $\PP^2$, a quadric $\PP^1 \times \PP^1$ or a ruled surface $\holom{f}{M}{B}$ over a curve of genus $g(B) \geq 1$.
\end{proposition}

\begin{proof} 
If $M$ is not uniruled, we conclude with Theorem \ref{theorem-not-uniruled}.
If $M$ is uniruled, it is a Mori fibre space by Corollary \ref{corollary-Mori-low-dimension}.
Thus if $M \not\simeq \PP^2$, it is a ruled surface $\holom{f}{M}{B}$.
If the curve $B$ is isomorphic to $\PP^1$ and $M$ is not a quadric, the Hirzebruch surface $M$ contains a rational curve $C$ such that $C^2<0$, in contradiction to 
Corollary  \ref{cor:minimal}.
\end{proof} 

\begin{remark} \label{remark-elliptic-bundles}
If $M$ is a compact K\"ahler surface, the arguments developed in the preceding sections show that $M$ is either an \'etale quotient of a torus, a Mori fibre space or a minimal surface of Kodaira dimension one. In the last case, using that $T_M$ is pseudoeffective by Corollary \ref{cor:extension} the techniques from \cite{HP20} allow to prove that, up to taking an \'etale cover, the surface $M$ is an elliptic bundle over a curve of genus at least two. If $M$ is not projective such an elliptic bundle does not trivialise after base change, so we can't use the technique from Theorem \ref{theorem-not-uniruled} to eliminate this case.
\end{remark}

\subsection{Negative curves} \label{subsection-algebraic-approach}

Theorem \ref{thm:mainone} can be applied to give some additional information 
on the tangent bundle:

\begin{lemma}  \label{lemma-pseff} 
Let $M$ be a compact K\"ahler manifold. Assume that $Z_M$ is Stein for some K\"ahler class. 
For every irreducible curve $C \subset M$, the restriction $T_M \vert_C$ is pseudoeffective.
 \end{lemma} 

\begin{proof} 
Let 
$$
0 \rightarrow \sO_M \rightarrow V \rightarrow T_M \rightarrow 0
$$
be the canonical extension such that $Z_M$ is Stein.
Then the closed subset $\PN(V_C) \setminus \PN(T_M \vert C)$ is Stein. Let $\holom{\nu}{\tilde C}{C}$ be the normalisation.
Then 
$$
\PN(\nu^* V_C) \setminus \PN(\nu^* T_M \vert C)
\quad \rightarrow \quad
\PN(V_C) \setminus \PN(T_M \vert C)
$$
is finite, so $\PN(\nu^* V_C) \setminus \PN(\nu^* T_M \vert C)$ is Stein \cite[V, \S 1, Thm.1 d)]{GR79}.
By Theorem \ref{thm:mainone} the normal bundle of $\PN(\nu^* T_M \vert C)$
in $\PN(\nu^* V_C)$ is pseudoeffective. Thus the normal bundle $\PN(T_M \vert C)$
in $\PN(V_C)$ is pseudoeffective. Yet this normal bundle is exactly the tautological class of 
$T_M \vert_C$, hence $T_M \vert _C$ is not pseudoeffective. 
\end{proof} 

We can now prove an analogue of Corollary \ref{cor:minimal} for curves of higher genus:

\begin{corollary} \label{corollary-negative}
Let $M$ be a compact K\"ahler surface. Assume that $M$ contains a smooth curve $C$
of genus at least two such that $C^2<0$. Then for any K\"ahler class $\alpha$, the canonical extension $Z_M$ is not Stein.
\end{corollary}

\begin{proof}
Assume that $Z_M$ is Stein.
By Lemma \ref{lemma-pseff} we know that the restriction $T_M\vert_{C}$ is pseudoeffective, so  by Lemma \ref{lemmapseffample}
its dual $\Omega_M \vert_{C}$ is not ample.

Yet since $g(C) \geq 2$ the normal bundle sequence 
$$   
0 \to T_{C} \to T_M \vert_{C} \to N_{C/M} \to 0 
$$
tells us that $\Omega_M \vert_{C}$ is ample, a contradiction.
\end{proof}

\subsection{Split tangent bundle} 

We have seen in Proposition \ref{proposition-gen-ample} that manifolds with split tangent bundle appear naturally in the analysis of the canonical extension. This situation can be studied through a technical lemma:

\begin{lemma} \label{lemma-split2}  
Let $M$ be a compact K\"ahler manifold such that $T_M \simeq \sF \oplus \sG$.
Let $\alpha $ be a K\"ahler class on $M$ and write 
$$ \alpha = \alpha_1 + \alpha_2 $$ 
with $\alpha_1 \in H^1(M,\sF^*) $ and $\alpha_2 \in H^1(M,\sG^*)$ according to the splitting. 
Let 
$$
0 \rightarrow \sO_M \rightarrow V_\sF \rightarrow \sF \rightarrow 0
$$
$$
0 \rightarrow \sO_M \rightarrow V_\sG \rightarrow \sG \rightarrow 0
$$
be the extensions defined by these cohomology classes, and set
$$
Z_\sF := \PP(V_\sF) \setminus \PP(\sF), \qquad
Z_\sG := \PP(V_\sG) \setminus \PP(\sG).
$$
Let $Z_M = Z_M^{\alpha}$ be the canonical extension. Then we have an isomorphism
$$
Z_M \simeq Z_\sF \times_M Z_\sG,
$$ 
so in fact 
$$
Z_M \simeq (\PP(V_\sF) \times_M \PP(V_\sG)) \setminus (\fibre{\tau_\sF}{\PP(\sF)} \cup \fibre{\tau_\sG}{\PP(\sG)})
$$
where $\tau_\sF$ and $\tau_\sG$ are the natural projections on the factors of the
fibre product $\PP(V_\sF) \times_M \PP(V_\sG)$.
\end{lemma} 

\begin{proof} We have an extension 
$$ 
0 \to \sO_M \to V_\sF \oplus V_\sG \to V_M \to  0
$$
as given in the following diagram 

\begin{equation} \label{big2}
\xymatrix{
0 \ar[r] & \sO_M \ar[d]^{=} \ar[r]^-{f \mapsto f \oplus f} &  \sO_M \oplus \sO_M  \ar[d]  \ar[r]^-{(\alpha, \beta) \mapsto \alpha - \beta}  & \sO_M \ar[d]  \ar[r] & 0
\\
0 \ar[r] & \sO_M  \ar[r] & V_\sF \oplus V_\sG \ar[r] \ar[d] & V_M \ar[r] \ar[d] & 0
\\
 & & \sF \oplus \sG \ar[r]^-{\simeq} & T_M  \ar[r] & 0
}
\end{equation}
Hence $$\mathbb P(V_M) \subset \mathbb P(V_\sF \oplus V_\sG).$$

Further, we have a morphism 
$$ \Phi: Z_\sF \times_M Z_\sG \to  \mathbb P(V_\sF \oplus V_\sG),$$
as follows.  Let $x \in M$, $u \in (Z_\sF)_x$ and $v \in (Z_\sG)_x$. 
Then 
$(u,v)$ is mapped to $[u \oplus v]$. 
Checking on the fibers over $x$, we see that 
$\Phi$ maps $Z_\sF \times_M Z_\sG$ isomorphically onto $Z_M$. 
\end{proof} 

For ruled surfaces this can be made more precise:

\begin{lemma} \label{lemma-split}  
Let $f: M = \mathbb P(\sE) \to B$ be a ruled surface where $\sE$ is a rank two vector bundle that is normalised in the sense of \cite[V, Ch.2]{Har77}. Assume that the tangent bundle splits such that $T_M \simeq T_{M/B} \oplus p^*(T_B)$. 
Let $\alpha $ be a K\"ahler class on $M$ and write 
$$ \alpha = \alpha_1 + \alpha_2 $$ 
with $\alpha_1 \in H^1(M,\Omega_{M/B}) $ and $\alpha_2 \in H^1(M,p^*(\Omega_B))$ according to the splitting. 
Then by Lemma \ref{lemma-split2} we have
$$ 
Z_M \simeq Z_{M/B} \times_B Z_B.
$$
and 
$$ 
Z_{M/B} \simeq \mathbb P(f^*\sE) \setminus \mathbb P (\zeta_M), 
$$
where $\zeta_M$ is the tautological bundle on $\PP(\sE)$.
\end{lemma} 

\begin{proof}  
The class $\alpha_1 \in H^1(M,\Omega_{M/B})$ defines an extension
$$
0 \rightarrow \sO_M \rightarrow V_{M/B} \rightarrow T_{M/B} \rightarrow 0.
$$
Since ${\rm Ext}^1(T_{M/B},\sO_M) = H^1(M,K_{M/B}) \simeq \mathbb C$, it follows from the relative Euler sequence that 
$$
V_{M/B} \simeq \zeta_M \otimes f^*(\sE^*).
$$
Thus, using $\sE^* \simeq \sE \otimes \det \sE^*$, 
we obtain  $\mathbb P(V_{M/B}) \simeq \mathbb P(f^*(\sE))$. 
Moreover since $T_{M/B} \simeq 2 \zeta_M \otimes f^* \det \sE^*$, this isomorphism
maps $\PP(T_{M/B})$ onto $\PP(\zeta_M)$.
\end{proof}

We can finally apply these argument to Serre's example \cite[\S 6.3]{Har70},\cite[\S 7]{Nee88}:

\begin{proposition} \label{prop:serre} Let $B$ be an elliptic curve and 
$$ 0 \to \sO_B \to \sE \to\sO_B \to 0 $$
a non split extension. Set $f: M := \mathbb P(\sE) \rightarrow B$ and choose any K\"ahler class on $M$. Then  $Z_M$ is Stein. 
\end{proposition}

\begin{proof}   
Since $M$ is almost homogeneous \cite{Pot69} the vector field on $B$ lifts to $M$,
so we have a splitting $T_M \simeq T_{M/B} \oplus \sO_M$.
Consider the commutative diagram
$$  \xymatrix{
 M \times_B Z_B \ar[r]^-{\rho}  \ar[d]^-{\sigma} & M \ar[d]^-{f} \\
Z_B \ar[r]^-{q} & B.}
$$
Then by  Lemma \ref{lemma-split}, we have
$$ 
Z_M \simeq \mathbb P(\rho^* f^*\sE) \setminus \mathbb P(\rho^*(\zeta_M)).
$$
Since $g(B) = 1$, the space $Z_B$ is Stein, so $H^1(Z_B,\sO_{Z_B}) = 0$. Thus
the exact sequence 
$$ 
0 \to \sO_{Z_B} \to q^*(\sE) \to \sO_{Z_B} \to 0 
$$
splits.
Hence
$$ M \times_B Z_B = \mathbb P(q^*(\sE)) \simeq \mathbb P_1 \times Z_B. $$
Further,
$$ 
\rho^* \pi^* (\sE) = \sigma^* q^*(\sE) \simeq \sO^{\oplus 2}_{M \times_B Z_B}
$$ 
so that 
$$ Z_M \simeq ( \mathbb P_1 \times \mathbb P_1 \times Z_B) \setminus \mathbb P(\rho^*(-K_M)).$$ 
Since $\rho^*(-K_M) = p_1^*(\sO_{\mathbb P_1})$, where $p_1: M \times_B Z_B = \mathbb P_1 \times Z_B$ is the projection, 
the divisor $\mathbb P(-K_M)$ cuts out in $(\mathbb P_1 \times \mathbb P_1) \times Z_B  \to Z_B$ in every fiber the same quadric 
hence $Z_M \simeq Q \times Z_B$ with an affine quadric $Q$. Thus $Z_M$ is Stein. 
\end{proof} 

\begin{corollary} Let $B$ be a elliptic curve, $\sE$ a semistable vector bundle of rank two on $B$ and $M = \PP(\sE)$.
Then $Z_M$ is Stein for any K\"ahler class on $M$.
\end{corollary}

\begin{proof} Since $\sE$ is semistable, we may assume that $\sE$ fits into an exact sequence
$$ 0 \to \sO_B \to \sE \to \sL \to 0, $$
where $\sL$ has degree $0$ or $1$. If $\sL$ has degree $0$, then either $\sL$ is torsion, and we are in a product situation after finite \'etale cover, or we are in
the Serre example, which is settled by Proposition \ref{prop:serre}. If $\sL$ has degree one, we may perform an \'etale cover of degree two and land in the Serre example. 
Hence $Z_M$ is Stein in all cases. 

\end{proof}

\begin{proof}[Proof of Theorem \ref{thm:maintwo}]
Assume first that $Z_M$ is Stein. By Proposition \ref{proposition-step} we are left to discuss the case of ruled surfaces $f: M \simeq \PP(\sE) \rightarrow B$ over a curve of genus at least one.
If $g(B) \geq 2$ we know by Corollary \ref{corollary-negative} that $M$ does not contain a curve with negative self-intersection. Yet by \cite[V, Prop.2.20]{Har77}
this implies that the invariant $e$ of the normalised vector bundle $\sE$ is at most $0$.
But by definition of $e$, see \cite[V, Prop.2.8]{Har77}, this means that $\sE$ is semistable.

Assume now that $Z_M$ is even affine. Then the tangent bundle $T_M$ is big by \cite[Prop.4.20]{GW20}, so $M$ is obviously not an \'etale quotient of a torus. By the first part of the statement we are left to exclude the case of ruled surfaces 
$f: M \simeq \PP(\sE) \rightarrow B$ over a curve of genus at least one.
If $g(B) \geq 2$ we know that $\sE$ is semistable, so $T_{M/B}$ is nef, but not big. Since $f^* T_B$ has negative degree, the tangent sequence shows that
$$
H^0(M, S^l T_M) \simeq H^0(M, S^l T_{M/B})
$$
for every $l \in \N$. Thus $T_M$ is not big.
So suppose $g(B) = 1$. If $\sE$ is semistable, then $T_M$ is nef. If $T_M$ would be big, then $c_1^2(M) > c_2(M)$, which is absurd. If $\sE$ is unstable, i.e., its invariant $e > 0$, we have
$\sE \simeq \sO_B \oplus L$ with $\deg L<0$. Let $B_0 \subset M$ be the section corresponding
to the negative quotient $\PP(L)$. Then we have $T_M|_{B_0} \simeq \sO_B \oplus L$, so 
the $T_M|_{B_0}$ is not big. If $Z_M$ was affine, its closed subset $Z_M|_{B_0}$ would also be affine. Yet by \cite[Prop.4.2]{GW20} this implies that $T_M|_{B_0}$ is big, a contradiction.
\end{proof}

\begin{remarks} 
While Theorem \ref{thm:maintwo} gives the expected characterisation of affineness of the canonical extension, the picture is less complete for the Stein property. Let us explain why the remaining cases are more subtle than one might expect:

1) Let $B$ be an elliptic curve, and let $\sE \simeq \sO_B \oplus L$ be a rank two vector bundle with $\deg L<0$. Then the restriction $T_M|_C$ is nef unless $C$ is the section $B_0 \subset M$ corresponding to the negative quotient $\PP(L)$. Making an elementary transformation in the sense of \cite{Mar73} along $B_0$ one can construct a rank two subbundle $K \subset T_M$ that is nef and big. 

2) Let $B$ be a smooth curve of genus $g \geq 2$ and $\sE$ a normalized locally free sheaf of rank two given by the exact sequence
$$  
0 \to \sO_B  \to \sE \to \sL \to 0. 
$$
Assume that the invariant $e$ is strictly negative,  i.e., $\sE$ is stable.

Let $f: M := \PP(\sE) \rightarrow B$ be the ruled surface, and 
denote by $\eta$ the extension class corresponding to
$$
0 \rightarrow T_{M/B} \rightarrow T_M \rightarrow f^* T_B \rightarrow 0.
$$
Denoting by $\sF$ the sheaf of traceless endomorphisms,  the extension class $\eta$ 
lives in 
$$ 
H^1(X, T_{M/B} \otimes p^*K_B) = H^1(B, \sF \otimes K_B) = H^0(B,\sF)=0.
$$
hence $\eta = 0$. Thus we have $T_M \simeq T_{M/B} \oplus f^* T_B$, so this situation fits into the framework of Lemma \ref{lemma-split}. Thus it might be a surprise that we were not able to prove that $Z_M$ is not Stein. Note however that in this case the subsheaf
$T_{M/B} \subset T_M$ is nef, so again $T_M$ is rather positive.
\end{remarks}

\section{The algebraic approach} \label{section-algebraic}

We start giving an algebraic proof of the following result, which is weaker than Theorem \ref{thm:mainone}, but has a more elementary proof. 

\begin{theorem} \label{thm:gennef} Let $M$ be a projective manifold of dimension $n$. 
Assume that $Z_M$ is Stein for some K\"ahler class. Then $\zeta_M$ is generically nef:
$$\zeta_M \cdot H_1 \cdot \ldots \cdot H_{2n-2} \geq 0$$
for all ample divisors $H_j$ on $\mathbb P(T_M)$. 
\end{theorem}

\begin{proof}[Proof of Theorem \ref{thm:gennef}]
We will use the terminology of $\mathbb Q$-twisted bundles as explained in \cite[Ch.6.2]{Laz04}.
By Corollary \ref{cor1}, it suffices to show that
$$ {\rm Nef}(\PN(V)) \to {\rm Nef}(\PNM) $$
is surjective. 
So let $\sL$ be a nef divisor class on $\PNM$, then we can write 
$$ 
\sL = a \zeta_M \otimes \pi_M^*(\mathcal A)
$$
with $a \in \N$ and $\mathcal A$ a divisor class on $M$. We may assume $a > 0$, otherwise there is nothing to prove. 

If $M=\PP^n$, the class $\zeta_V$ is ample, so we can assume $M \not\simeq \PP^n$.
We claim that $\mathcal A$ is nef. This implies the statement:
$a \zeta_M + \pi_M^*(\mathcal A)$ being nef is equivalent saying that the
$\mathbb Q$-bundle $T_M \otimes \frac{1}{a} \mathcal A$ is nef. 
Thus we have an extension of nef $\mathbb Q$-vector bundles
$$
0 \rightarrow \mathcal A \rightarrow V \otimes \frac{1}{a} \mathcal A \rightarrow T_M \otimes \frac{1}{a} \mathcal A \rightarrow 0,
$$
hence $V \otimes \frac{1}{a} \mathcal A$ is nef \cite[Thm.6.2.12,(ii),(v)]{Laz04}.

{\em Proof of the claim.}
 By the cone theorem, it suffices to show that $\mathcal A \cdot C \geq 0 $ for either $C$ a curve generating a $K_M$-negative extremal ray and for curves $C$ with $K_M \cdot C \geq 0$. 
\begin{itemize}
\item If $K_M \cdot C \geq 0$, the vector bundle $T_M \vert_C$ cannot be ample. 
If $T_M \vert_C$ is not nef, then $\deg \mathcal A|_C>0$ by Hartshorne's criterion \cite[Thm.6.4.15]{Laz04}.
If $T_M \vert_C$ is nef, then $c_1(T_M|_C) \leq 0$ implies that $T_M \vert_C$ is numerically flat.
Thus $\deg \mathcal A|_C \geq 0$.
\item If $C$ is an extremal rational curve with $0 < -K_M \cdot C \leq n$,
denote by $\holom{\nu}{\PP^1}{M}$ its normalisation. 
Then $\nu^* T_M$ has a trivial quotient since
it contains the image of $\sO_{\PP^1}(2) \rightarrow \nu^* T_M$. Thus again $\deg \nu^* \mathcal A = \deg
\mathcal A|_C \geq 0$. 
\item If $C$ is an extremal rational curve with $-K_M \cdot C = n+1$, then $M = \PN^n$ by \cite{CMS02}, which we excluded.
\end{itemize}
\end{proof} 

\begin{corollary} \label{cor2} Let $M$ be a projective manifold. 
Assume $Z_M$ Stein and that $\overline{\mathcal {CI}}(\PNM) = \overline{NM}(\PNM)$. Then $\zeta_M$ is pseudoeffective. 
\end{corollary} 

\begin{proof} By Theorem \ref{thm:gennef}, the class $\zeta_M$ is generically nef. By our assumption and \cite{BDPP13}, the class $\zeta_M$ is then pseudoeffective. 
\end{proof} 

We give some first evidence in favour of Conjecture \ref{conjecture-gen-nef}:

\begin{proposition} \label{prop:K3}  Let $M$ be a projective K3 surface with $\rho(M) = 1$ Then $\zeta_M$ is not generically nef.
\end{proposition} 

By Theorem \ref{thm:gennef} this shows that $Z_M$ is never Stein for K3 surfaces $M$ with $\rho(M) = 1$. 

\begin{proof} 
Let $H$ be the ample generator the Picard group, and 
write $H^2 = d$. 
Let $a$ be the unique positive real number such that $\zeta_M + a \pi_M^* H$ is nef but not ample. Then $\zeta_M$ is generically nef precisely when
$$ \zeta_M \cdot (\zeta + a \pi_M^* H)^2 \geq 0.$$ 
Now $ \zeta_M \cdot (\zeta + a \pi_M^* H)^2 = - 24 + a^2 d$. If $d > 2$, then  $a^2 d < 24$ by \cite[Prop.3.2]{GO20}, 
thus 
$$ \zeta_M \cdot (\zeta + a \pi_M^* H)^2 < 0.$$ 
 In case $d = 2$, $a = 3$ by \cite[Sect. 4.1]{GO20},
so the inequality holds, too, and $\zeta_M$ is not generically nef. 

\end{proof} 

Finally, we present a completely different approach for K3-surfaces $M$ containing a nodal curve to show that $Z_M$ is not Stein.  We start with some preparations.

\begin{lemma} \label{lemmafinitecover}
Let $M$ be a smooth projective surface, and $T$ a rank two vector bundle on $M$. Let $C \subset \PP(T)$ be a smooth curve that is not a fibre of
$\holom{\pi}{\PP(T)}{M}$. Then there exists a smooth surface $S \subset \PP(T)$ such that $C \subset S$ and $\pi|_S$ is finite.  
\end{lemma}

\begin{proof}
We can assume without loss of generality that $T$ is ample, and denote by
$\zeta$ the tautological class of $\PP(T)$. By \cite[Cor.1.7.5]{BS95} we know that for $t \gg 0$ a general element $S \in | \sI_C \otimes \sO_{\PP(T)}(t\zeta) |$
is smooth, so we are left to show that $\pi|_S$ is finite.
Using the fact the $t$-th relative Segre embedding embeds
the fibres of $\PP(T)$ as linearly non-degenerate curves in $\PP(S^t T)$, this is equivalent to showing that there exists a global section of
$\pi_*(\sI_C \otimes \sO_{\PP(T)}(t\zeta))$ that does not vanish in any point of $S$.

For $t \gg 0$ the sheaf $\sI_C \otimes \sO_{\PP(T)}(t\zeta)$ is globally generated,
so its direct image $\pi_*(\sI_C \otimes \sO_{\PP(T)}(t\zeta))$ is globally generated. Let $V$ be its space of global sections, then the morphism
$$
V \otimes \sO_{M \setminus \pi(C)} \rightarrow S^t T \otimes \sO_{M \setminus \pi(C)}
$$
is surjective. Since $t>1$ the vector bundle $S^t T$ has rank at least three, so by \cite[II, Ex.8.2]{Har77} a general element of $V$ defines a section
of $S^t T \otimes \sO_{M \setminus \pi(C)}$ that does not vanish on any point of $M \setminus \pi(C)$. Thus we are left to show the statement
over $\pi(C)$: 
denote by $\holom{\nu}{C}{S}$ the restriction of $\pi$ to $C$, and by
$\holom{\tilde \nu}{\PP(\nu^* T)}{\PP(T)}$ the natural map. Let $\tilde C \subset \PP(\nu^* T)$
be the natural section such that $\tilde \nu(\tilde C)=C$. Since $\zeta$ is ample we know that
$\tilde \nu_* (\sI_{\tilde C}) \otimes \sO_{\PP(T)}(t\zeta) \simeq \sI_C \otimes \sO_{\PP(T)}(t\zeta)$ is globally generated for $t \gg 0$, in particular 
$\sO_{\PP(\nu^* V)}(- \tilde C) \otimes \tilde \nu^* \sO_{\PP(T)}(t\zeta)$ is ample and generated by pull-backs
of global sections for $t \gg 0$. Thus a general element of the linear subsystem of
$|\tilde \nu^* t \zeta - \tilde C|$ generated by the pull-backs is a smooth and irreducible curve $R$.
Thus the intersection of $S$ with $\fibre{\pi}{\pi(C)}$ has exactly two irreducible components, 
the curve $C$ and $\tilde \nu(R)$, none of them is a fibre.
\end{proof}

\begin{lemma} \label{lemmalocalvectorbundles}
Let $U$ be a smooth analytic surface that contains a rational curve $\PP^1 \simeq C \subset U$
such that $C^2:=-k \leq -1$ and $U$ is a deformation retract of $C$\footnote{In particular one has $\pic(U) \simeq \pic(C) \simeq \Z$}.
For every $l \in \Z$, we denote by $\sO_U(l)$ the unique line bundle such that $\sO_U(l) \otimes \sO_C \simeq \sO_{\PP^1}(l)$. 
Then the following holds:
\begin{itemize}
\item One has an exact sequence
$$
0 \rightarrow \sO_U(-k) \rightarrow T_U \rightarrow \sO_U(2) \rightarrow 0.
$$
\item If $k \leq -3$, let $T$ be a vector bundle on $U$ such that $\det T \simeq \sO_U$ and 
$T \otimes \sO_C \simeq \sO_{\PP^1}(2) \oplus \sO_{\PP^1}(-2)$. Then one has
$$
T \simeq \sO_U(2) \oplus \sO_U(-2).
$$
\end{itemize}
\end{lemma}

\begin{proof}
For the first statement we observe that $U$ is isomorphic to a neighbourhood of
the negative section in the Hirzebruch surface $\psi: \mathbb{F}_k \simeq \PP(\sO_{\PP^1} \oplus \sO_{\PP^1}(-k))
\rightarrow \PP^1$. The exact sequence is then obtained as the restriction of the 
tangent sequence
$$
0 \rightarrow T_{\mathbb{F}_k/\PP^1} \rightarrow T_{\mathbb{F}_k} \rightarrow \psi^* T_{\PP^1} \rightarrow 0
$$
to this neighbourhood.

For the second statement, let $C = \mathbb P_1$ and consider the $l-$th infinitesimal neighborhood $C_l$ of $C$ in $U$. Then the extensions of $T_{\vert C_l}$ to $C_{l+1}$ are parametrized by 
$$ H^1(C, (T_{\vert C}^* \otimes T_{\vert C} \otimes S^kN^*_{C/U}), $$ 
see e.g. \cite{Pe81}. 
This group equals 
$$
 H^1(C, (\sO_C(4) \otimes \sO_C^{\oplus 2} \otimes \sO_C(-4)) \otimes \sO_C(3k)) = 0.
$$
Hence $T_{\vert C}$ has a unique extension to all infinitesimal neighborhoods $C_k$, namely 
$\sO_{C_k}(2) \oplus \sO_{C_k}(-2)$. Thus 
by \cite{Pe81}, there exists an open neighborhhood $V \subset U$ of $C$ such that 
$$ T_{\vert V} \simeq \sO_V(2) \oplus \sO_V(-2).$$

\end{proof}

\begin{theorem} \label{theoremnodal}
Let $M$ be a projective K3-surface containing a nodal rational curve. 
Then $Z_M$ is not Stein.
\end{theorem} 

\begin{proof}
We argue by contradiction and assume that $Z_M$ is Stein.

Choose a nodal rational curve $f: \PP^1 \rightarrow M$. Since $f$ is immersive, we have a canonical surjection
$f^* \Omega_M \rightarrow \Omega_{\PP^1}$. Since the image of $f$ is nodal, the induced map $\tilde f: \PP^1 \rightarrow \PP(T_M)$ is embedding, and we denote by $C$ its image.
By Lemma \ref{lemmafinitecover} there exists a surface $S \subset \PP(T_M)$ be a surface 
containing $C$ such that the morphism $\holom{g:=\pi|_S}{S}{M}$ is finite. The pull-back of the canonical extension 
gives an exact sequence of vector bundles  
\begin{equation} \label{pullbackg}
0 \rightarrow \sO_S \rightarrow g^* V \rightarrow g^* T_M \rightarrow 0
\end{equation}
Since $g$ is finite and $\PP(V) \setminus \PP(T_M)$ is Stein, the space
$\PP(g^* V) \setminus \PP(g^* T_M)$ is also Stein \cite[V, \S 1, Thm.1 d)]{GR79}.

Since the restriction of $g$ to $C$ is the normalisation of $f(\PP^1)$
one has 
$$
g^* T_M \otimes \sO_C \simeq T_C \oplus \sO_C(-2) \simeq \sO_C(2) \oplus \sO_C(-2).
$$ 
By the ramification formula one has $K_S \cdot C \geq 0$, hence $-k:=C^2 \leq -2$ with
equality if and only if $g$ is \'etale in a neighbourhood of $C$.
Thus the smooth rational curve $C$ is contractible, and we denote by $\holom{\varphi}{S}{S'}$ its contraction. We claim that the exact sequence \eqref{pullbackg} has the extension property near $C$,
by Lemma \ref{lemmaextensionproperty} this gives the desired contradiction.

If $C^2=-2$ the tangent map $T_S \rightarrow g^* T_M$ is an isomorphism near 
$C$, so the exact sequence \eqref{pullbackg} identifies to the canonical extension of $S$
induces by the K\"ahler class $g^* \alpha$. In this case the extension property follows as in
the proof of the last case of Theorem \ref{theorem-contractions}, so for the rest of the proof we focus on the (more difficult) case $k \leq -3$.

Let $U \subset S$ be an analytic neighbourhood of $C$ that is a deformation retract of $C$.
Then using the notation of Lemma \ref{lemmalocalvectorbundles} one has $g^* T_M \simeq \sO_U(2) \oplus \sO_U(-2)$. Observe that
$$
\mbox{Ext}^1(\sO_U(2) \oplus \sO_U(-2), \sO_U) 
\simeq
H^1(U, \sO_U(-2) \oplus \sO_U(+2))
\simeq 
H^1(\PP^1, \omega_{\PP^1})
$$
where in the last step we used that the inclusion $\sO_U(2) \rightarrow g^* T_M$
restricts to the natural inclusion $T_{\PP^1} \rightarrow f^* T_M$ on $C$.
Thus the restriction of \eqref{pullbackg} to $U$ is isomorphic to an exact sequence
\begin{equation} \label{exactsequenceU}
0 \rightarrow \sO_U \rightarrow \mathcal E_U \oplus  \sO_U(-2)  \rightarrow
\sO_U(2) \oplus \sO_U(-2) \rightarrow 0
\end{equation}
where the extension 
\begin{equation} \label{EulerU}
0 \rightarrow \sO_U \rightarrow \mathcal E_U \rightarrow
\sO_U(2) \rightarrow 0
\end{equation}
restricts to the Euler sequence on $C$. 

Denote by $\holom{\varphi_U}{U}{U'}$ the restriction of the contraction to $U$.
Since $U \rightarrow U'$ contracts a smooth rational curve onto a point, the surface $U'$ has a Hirzebruch-Jung singularity in $\varphi(C)$, in particular it is a quotient singularity.

Since the line bundle $\sO_U(-2)$ has negative degree on the exceptional divisor $C$,
we have a chain of inclusion $\sO_U(C) \simeq \sO_{U}(-k) \rightarrow \sO_U(-2) \rightarrow \sO_U$.
Pushing forward to $U'$ we obtain $\sO_{U'} \rightarrow (\varphi_U)_* \sO_U(-2) \rightarrow \sO_{U'}$, so $(\varphi_U)_* \sO_U(-2) \simeq \sO_{U'}$. Since $R^1 (\varphi_U)_* \sO_U=0$
The push-forward of the exact sequence \eqref{exactsequenceU} is the exact sequence
$$
0 \rightarrow \sO_{U'} \rightarrow (\varphi_U)_* \mathcal E_U \oplus \sO_{U'} 
\rightarrow  (\varphi_U)_* \sO_U(2)  \oplus \sO_{U'}  \rightarrow 0.
$$
Thus it is sufficient to check the extension property in Lemma \ref{lemmaextensionproperty}
for the sequence \eqref{EulerU}. The idea is now to reduce to the case of a canonical extension on $U$:
by Lemma \ref{lemmalocalvectorbundles} we have an exact sequence
$$
0 \rightarrow \sO_U(-k) \rightarrow T_U \rightarrow \sO_U(2) \rightarrow 0.
$$
Since $H^1(U, \sO_U(k))=0$ the long exact sequence in cohomology shows that
$$
\mbox{Ext}^1(T_U, \sO_U) \simeq \mbox{Ext}^1(\sO_U(2), \sO_U) \simeq H^1(\PP^1, \omega_{\PP^1}).
$$
Thus if   
\begin{equation} \label{canonicalU}
0 \rightarrow \sO_U \rightarrow V_U \rightarrow T_U \rightarrow 0
\end{equation}
is  the canonical extension defined by $g^* \alpha|_U$, the morphism 
$\sO_U(-k) \rightarrow T_U$
lifts to a morphism $\beta_U : \sO_U(-k) \rightarrow V_U$ such that the exact sequence 
$$
0 \rightarrow \sO_U \rightarrow V_U/\sO_U(-k) \rightarrow T_U/\sO_U(-k) \simeq \sO_U(2) \rightarrow 0
$$
is isomorphic to $\eqref{EulerU}$.

Let now $\holom{p}{\tilde U}{U'}$ be the quasi-\'etale cover as in Lemma \ref{lemmaextensionproperty}.
Pushing down to $U'$, and taking the reflexive pull-back to $\tilde U$ we obtain a commutative diagram:
\begin{equation} \label{bigdiagram}
\xymatrix{
& & & 0 \ar[d] & 
\\
& & & \sO_{\tilde U} \ar[ld]_j \ar[d]
\\
0 \ar[r] & \sO_{\tilde U} \ar[d] \ar[r]^-{i_{V_U}} & p^{[*]} (\varphi_U)_* V_U \ar[r] \ar[d] & p^{[*]} (\varphi_U)_* T_U \simeq T_{\tilde U} \ar[r] \ar[d] & 0
\\
0 \ar[r] & \sO_{\tilde U} \ar[r]^-{i_{\mathcal E_U}} & p^{[*]} (\varphi_U)_* \mathcal E_U \ar[r] & p^{[*]} (\varphi_U)_* \sO_U(2) \ar[r] & 0
}
\end{equation}
Since the extension property holds for \eqref{canonicalU}, we already know that the morphism $i_{V_U}$ is injective in every point, and we want to see that the same holds for $i_{\mathcal E_U}$. Now observe that since $U \rightarrow U'$ is a functorial resolution of $U'$, one has $(\varphi_U)_* T_U \simeq T_{U'}$. In particular the exact sequence 
$$
0 \rightarrow \sO_U \rightarrow (\varphi_U)_* V_U \rightarrow (\varphi_U)_* T_U \rightarrow 0
$$
is locally split and the image of $\sO_{U'} \simeq (\varphi_U)_* \sO_U(-k) \rightarrow  (\varphi_U)_* V_U$ is contained in the image of the splitting map. This property is preserved under the reflexive pull-back to $\tilde U$, i.e. for every point of $\tilde U$ the image of the injection $i_{V_U}$
is not contained the image of $j$. Since $i_{\mathcal E_U}$ is obtained from $i_{V_U}$ by taking the quotient by $j_* \sO_{\tilde U}$, it is injective in every point.
\end{proof}

\begin{remark*}
Let $M$ be a projective K3 surface with Picard number $\rho(M) \geq 5$. Then $M$ 
admits an elliptic fibration \cite[Ch.11, Prop.1.3]{Huy16}. If $Z_M$ is Stein, Theorem \ref{theoremnodal} in connection with the non-existence of $(-2)$-curves shows that all the singular fibres are cuspidal elliptic curves.
Looking at the Weierstra\ss\ model of the associated Jacobian fibration and using 
\cite[p.40]{Mir89} one can then show
that the elliptic fibration is isotrivial with general fibre the elliptic curve with $J$-invariant $0$. Thus Theorem \ref{theoremnodal} covers almost all the K3 surfaces with
$\rho(M) \geq 5$. 
\end{remark*}



\begin{thebibliography}{BCHM10}

\bibitem[AS60]{AS60}
Aldo Andreotti and Wilhelm Stoll, \emph{Extension of holomorphic maps}, Ann. of
  Math. (2) \textbf{72} (1960), 312--349. \MR{123734}

\bibitem[AT82]{AT82}
Vincenzo Ancona and Giuseppe Tomassini, \emph{Modifications analytiques},
  Lecture Notes in Mathematics, vol. 943, Springer-Verlag, Berlin, 1982.
  \MR{MR673560 (84g:32022)}

\bibitem[BCHM10]{BCHM10}
Caucher Birkar, Paolo Cascini, Christopher~D. Hacon, and James McKernan,
  \emph{Existence of minimal models for varieties of log general type}, J.
  Amer. Math. Soc. \textbf{23} (2010), no.~2, 405--468. \MR{2601039}

\bibitem[BDPP13]{BDPP13}
S{\'e}bastien Boucksom, Jean-Pierre Demailly, Mihai P{\u a}un, and Thomas
  Peternell, \emph{The pseudo-effective cone of a compact {K\"a}hler manifold
  and varieties of negative {K}odaira dimension}, Journal of Algebraic Geometry
  \textbf{22} (2013), 201--248.

\bibitem[BS95]{BS95}
Mauro~C. Beltrametti and Andrew~J. Sommese, \emph{The adjunction theory of
  complex projective varieties}, de Gruyter Expositions in Mathematics,
  vol.~16, Walter de Gruyter \& Co., Berlin, 1995. \MR{MR1318687 (96f:14004)}

\bibitem[CMSB02]{CMS02}
Koji Cho, Yoichi Miyaoka, and Nicholas~I. Shepherd-Barron,
  \emph{Characterizations of projective space and applications to complex
  symplectic manifolds}, Higher dimensional birational geometry ({K}yoto,
  1997), Adv. Stud. Pure Math., vol.~35, Math. Soc. Japan, Tokyo, 2002,
  pp.~1--88. \MR{1929792 (2003m:14080)}

\bibitem[CP91]{CP91}
Fr{\'e}d{\'e}ric Campana and Thomas Peternell, \emph{Projective manifolds whose
  tangent bundles are numerically effective}, Math. Ann. \textbf{289} (1991),
  no.~1, 169--187. \MR{1087244 (91m:14061)}

\bibitem[Cut88]{Cut88}
Steven Cutkosky, \emph{Elementary contractions of {G}orenstein threefolds},
  Math. Ann. \textbf{280} (1988), no.~3, 521--525. \MR{936328 (89k:14070)}

\bibitem[Deb01]{Deb01}
Olivier Debarre, \emph{Higher-dimensional algebraic geometry}, Universitext,
  Springer-Verlag, New York, 2001. \MR{MR1841091 (2002g:14001)}

\bibitem[Dem12]{Dem12}
Jean-Pierre Demailly, \emph{Analytic methods in algebraic geometry}, Surveys of
  Modern Mathematics, vol.~1, International Press, Somerville, MA; Higher
  Education Press, Beijing, 2012. \MR{2978333}

\bibitem[DPS96]{DPS96}
Jean-Pierre Demailly, Thomas Peternell, and Michael Schneider,
  \emph{Holomorphic line bundles with partially vanishing cohomology},
  Proceedings of the {H}irzebruch 65 {C}onference on {A}lgebraic {G}eometry
  ({R}amat {G}an, 1993), Israel Math. Conf. Proc., vol.~9, Bar-Ilan Univ.,
  Ramat Gan, 1996, pp.~165--198. \MR{1360502}

\bibitem[Dru18]{Dru18}
St\'{e}phane Druel, \emph{A decomposition theorem for singular spaces with
  trivial canonical class of dimension at most five}, Invent. Math.
  \textbf{211} (2018), no.~1, 245--296, doi: 10.1007/s00222-017-0748-y.
  \MR{3742759}

\bibitem[Fuj04]{Fuj04}
Osamu Fujino, \emph{Termination of 4-fold canonical flips}, Publ. Res. Inst.
  Math. Sci. \textbf{40} (2004), no.~1, 231--237. \MR{2030075}

\bibitem[Fuj05]{Fuj05}
\bysame, \emph{Addendum to: ``{T}ermination of 4-fold canonical flips''
  [{P}ubl. {R}es. {I}nst. {M}ath. {S}ci. {\bf 40} (2004), no. 1, 231--237;
  mr2030075]}, Publ. Res. Inst. Math. Sci. \textbf{41} (2005), no.~1, 251--257.
  \MR{2115973}

\bibitem[Ful11]{Fu11}
Mihai Fulger, \emph{The cones of effective cycles on projective bundles over
  curves}, Math. Z. \textbf{269} (2011), no.~1-2, 449--459. \MR{2836078}

\bibitem[GO20]{GO20}
Frank Gounelas and John~Christian Ottem, \emph{Remarks on the positivity of the
  cotangent bundle of a {K}3 surface}, \'{E}pijournal G\'{e}om. Alg\'{e}brique
  \textbf{4} (2020), Art. 8, 16. \MR{4149966}

\bibitem[Goo69]{Go69}
Jacob~Eli Goodman, \emph{Affine open subsets of algebraic varieties and ample
  divisors}, Ann. of Math. (2) \textbf{89} (1969), 160--183. \MR{242843}

\bibitem[GR79]{GR79}
Hans Grauert and Reinhold Remmert, \emph{Theory of {S}tein spaces}, Grundlehren
  der Mathematischen Wissenschaften, vol. 236, Springer-Verlag, Berlin-New
  York, 1979, Translated from the German by Alan Huckleberry. \MR{580152}

\bibitem[GS21]{GS21}
Patrick Graf and Martin Schwald, \emph{The {K}odaira problem for {K}\"{a}hler
  spaces with vanishing first {C}hern class}, Forum Math. Sigma \textbf{9}
  (2021), Paper No. e24, 15. \MR{4235163}

\bibitem[GW20]{GW20}
Daniel Greb and Michael~Lennox Wong, \emph{Canonical complex extensions of
  {K}\"{a}hler manifolds}, J. Lond. Math. Soc. (2) \textbf{101} (2020), no.~2,
  786--827. \MR{4093975}

\bibitem[Har70]{Har70}
Robin Hartshorne, \emph{Ample subvarieties of algebraic varieties}, Notes
  written in collaboration with C. Musili. Lecture Notes in Mathematics, Vol.
  156, Springer-Verlag, Berlin, 1970. \MR{0282977 (44 \#211)}

\bibitem[Har77]{Har77}
\bysame, \emph{Algebraic geometry}, Springer-Verlag, New York, 1977, Graduate
  Texts in Mathematics, No. 52. \MR{MR0463157 (57 \#3116)}

\bibitem[Hir75]{Hir75}
Heisuke Hironaka, \emph{Flattening theorem in complex-analytic geometry}, Amer.
  J. Math. \textbf{97} (1975), 503--547. \MR{393556}

\bibitem[HP15]{HoPe2}
Andreas H{\"o}ring and Thomas Peternell, \emph{Mori fibre spaces for {K}\"ahler
  threefolds}, J. Math. Sci. Univ. Tokyo \textbf{22} (2015), no.~1, 219--246.

\bibitem[HP16]{HP16}
\bysame, \emph{Minimal models for {K}\"ahler threefolds}, Invent. Math.
  \textbf{203} (2016), no.~1, 217--264. \MR{3437871}

\bibitem[HP19]{HP19}
Andreas H\"{o}ring and Thomas Peternell, \emph{Algebraic integrability of
  foliations with numerically trivial canonical bundle}, Invent. Math.
  \textbf{216} (2019), no.~2, 395--419. \MR{3953506}

\bibitem[HP20]{HP20}
Andreas H{\"o}ring and Thomas Peternell, \emph{A nonvanishing conjecture for
  cotangent bundles}, arXiv preprint \textbf{2006.05225} (2020).

\bibitem[Huy16]{Huy16}
Daniel Huybrechts, \emph{Lectures on {K}3 surfaces}, Cambridge Studies in
  Advanced Mathematics, vol. 158, Cambridge University Press, Cambridge, 2016.
  \MR{3586372}

\bibitem[Kaw89]{Kaw89}
Yujiro Kawamata, \emph{Small contractions of four-dimensional algebraic
  manifolds}, Math. Ann. \textbf{284} (1989), no.~4, 595--600. \MR{MR1006374
  (91e:14039)}

\bibitem[Kaw92]{Kaw92}
\bysame, \emph{Termination of log flips for algebraic {$3$}-folds}, Internat.
  J. Math. \textbf{3} (1992), no.~5, 653--659. \MR{1189678}

\bibitem[KK83]{Kau83}
Ludger Kaup and Burchard Kaup, \emph{Holomorphic functions of several
  variables}, de Gruyter Studies in Mathematics, vol.~3, Walter de Gruyter \&
  Co., Berlin, 1983, An introduction to the fundamental theory, With the
  assistance of Gottfried Barthel, Translated from the German by Michael
  Bridgland. \MR{MR716497 (85k:32001)}

\bibitem[KM98]{KM98}
J{\'a}nos Koll{\'a}r and Shigefumi Mori, \emph{Birational geometry of algebraic
  varieties}, Cambridge Tracts in Mathematics, vol. 134, Cambridge University
  Press, Cambridge, 1998, With the collaboration of C. H. Clemens and A. Corti.
  \MR{MR1658959 (2000b:14018)}

\bibitem[KMM94]{KMM94}
Sean Keel, Kenji Matsuki, and James McKernan, \emph{Log abundance theorem for
  threefolds}, Duke Math. J. \textbf{75} (1994), no.~1, 99--119. \MR{1284817}

\bibitem[KMM04]{KMM94b}
\bysame, \emph{Corrections to: ``{L}og abundance theorem for threefolds''
  [{D}uke {M}ath. {J}. {\bf 75} (1994), no. 1, 99--119; mr1284817]}, Duke Math.
  J. \textbf{122} (2004), no.~3, 625--630. \MR{2057020 (2005a:14018)}

\bibitem[Kol97]{Kol97}
J\'{a}nos Koll\'{a}r, \emph{Singularities of pairs}, Algebraic
  geometry---{S}anta {C}ruz 1995, Proc. Sympos. Pure Math., vol.~62, Amer.
  Math. Soc., Providence, RI, 1997, pp.~221--287. \MR{1492525}

\bibitem[Kos82]{Ko82}
Siegmund Kosarew, \emph{Ein {V}erschwindungssatz f\"{u}r gewisse
  {K}ohomologiegruppen in {U}mgebung kompakter komplexer {U}nterr\"{a}ume mit
  konvex/konkavem {N}ormalenb\"{u}ndel}, Math. Ann. \textbf{261} (1982), no.~3,
  315--326. \MR{679793}

\bibitem[KP90]{KP90}
Siegmund Kosarew and Thomas Peternell, \emph{Formal cohomology, analytic
  cohomology and nonalgebraic manifolds}, Compositio Math. \textbf{74} (1990),
  no.~3, 299--325. \MR{1055698}

\bibitem[Laz04]{Laz04}
Robert Lazarsfeld, \emph{Positivity in algebraic geometry. {I,II}}, Ergebnisse
  der Mathematik und ihrer Grenzgebiete., vol.~48, Springer-Verlag, Berlin,
  2004. \MR{MR2095471}

\bibitem[LPT18]{LPT18}
Frank Loray, Jorge~Vit\'{o}rio Pereira, and Fr\'{e}d\'{e}ric Touzet,
  \emph{Singular foliations with trivial canonical class}, Invent. Math.
  \textbf{213} (2018), no.~3, 1327--1380. \MR{3842065}

\bibitem[Mar73]{Mar73}
Masaki Maruyama, \emph{On a family of algebraic vector bundles}, Number theory,
  algebraic geometry and commutative algebra, in honor of {Y}asuo {A}kizuki,
  1973, pp.~95--146. \MR{0360587}

\bibitem[Mat13]{Mat13}
Shin-ichi Matsumura, \emph{Asymptotic cohomology vanishing and a converse to
  the {A}ndreotti-{G}rauert theorem on surfaces}, Ann. Inst. Fourier (Grenoble)
  \textbf{63} (2013), no.~6, 2199--2221. \MR{3237444}

\bibitem[Mir89]{Mir89}
Rick Miranda, \emph{The basic theory of elliptic surfaces}, Dottorato di
  Ricerca in Matematica. [Doctorate in Mathematical Research], ETS Editrice,
  Pisa, 1989. \MR{1078016}

\bibitem[Mok82]{Mok81}
Ngaiming Mok, \emph{The {S}erre problem on {R}iemann surfaces}, Math. Ann.
  \textbf{258} (1981/82), no.~2, 145--168. \MR{641821}

\bibitem[Mor82]{Mor82}
Shigefumi Mori, \emph{Threefolds whose canonical bundles are not numerically
  effective}, Ann. of Math. (2) \textbf{116} (1982), no.~1, 133--176.
  \MR{MR662120 (84e:14032)}

\bibitem[Nee88]{Nee88}
Amnon Neeman, \emph{Steins, affines and {H}ilbert's fourteenth problem}, Ann.
  of Math. (2) \textbf{127} (1988), no.~2, 229--244. \MR{932296}

\bibitem[Ott12]{Ott12}
John~Christian Ottem, \emph{Ample subvarieties and {$q$}-ample divisors}, Adv.
  Math. \textbf{229} (2012), no.~5, 2868--2887. \MR{2889149}

\bibitem[Pet81]{Pe81}
Thomas Peternell, \emph{Vektorraumb\"{u}ndel in der {N}\"{a}he von kompakten
  komplexen {U}nterr\"{a}umen}, Math. Ann. \textbf{257} (1981), no.~1,
  111--134. \MR{630650}

\bibitem[Pet11]{Pe11}
\bysame, \emph{Generically nef vector bundles and geometric applications},
  Complex and differential geometry, Springer Proc. Math., vol.~8, Springer,
  Heidelberg, 2011, pp.~345--368. \MR{2964482}

\bibitem[Pot69]{Pot69}
Joseph Potters, \emph{On almost homogeneous compact complex analytic surfaces},
  Invent. Math. \textbf{8} (1969), 244--266. \MR{259166}

\bibitem[PT13]{PT13}
Jorge~Vit\'orio Pereira and Fr\'ed\'eric Touzet, \emph{Foliations with
  vanishing {C}hern classes}, Bull. Braz. Math. Soc. (N.S.) \textbf{44} (2013),
  no.~4, 731--754. \MR{3167130}

\bibitem[Sch73]{Sch73}
Michael Schneider, \emph{\"{U}ber eine {V}ermutung von {H}artshorne}, Math.
  Ann. \textbf{201} (1973), 221--229. \MR{357858}

\bibitem[Tou08]{Tou08}
Fr\'ed\'eric Touzet, \emph{Feuilletages holomorphes de codimension un dont la
  classe canonique est triviale}, Ann. Sci. \'Ec. Norm. Sup\'er. (4)
  \textbf{41} (2008), no.~4, 655--668. \MR{2489636}

\bibitem[Yan19]{Yan19}
Xiaokui Yang, \emph{A partial converse to the {A}ndreotti-{G}rauert theorem},
  Compos. Math. \textbf{155} (2019), no.~1, 89--99. \MR{3878570}

\end{thebibliography}

\providecommand{\bysame}{\leavevmode\hbox to3em{\hrulefill}\thinspace}
\providecommand{\MR}{\relax\ifhmode\unskip\space\fi MR }
\providecommand{\MRhref}[2]{%
  \href{http://www.ams.org/mathscinet-getitem?mr=#1}{#2}
}
\providecommand{\href}[2]{#2}

\end{document}